\documentclass[12pt]{amsart}
\usepackage{amsfonts,amsthm,amsmath,amscd,amssymb}
\usepackage{t1enc}
\usepackage[mathscr]{eucal}
\usepackage{xcolor}
\usepackage{tikz}
\usepackage{tikz-cd}
\usepackage{epic}
\usepackage{hyperref}
\numberwithin{equation}{section}
\usepackage[margin=2.9cm]{geometry}
\usepackage{cite}
\usepackage{parskip}
\usepackage[f]{esvect}
\usepackage{caption}
\usepackage{subcaption}
\usepackage{array}
\usepackage{booktabs}
\usepackage{enumerate}
\usepackage{mathtools}
\usepackage{adjustbox}
\usepackage{listings}
\usepackage{forest,array}
\usetikzlibrary{shadows}

\newenvironment{talign*}
 {\csname align*\endcsname}
 {\endalign}

\theoremstyle{plain}
\newtheorem{thm}{Theorem}[section]
\newtheorem{lem}[thm]{Lemma}
\newtheorem{cor}[thm]{Corollary}
\newtheorem{prop}[thm]{Proposition}

\newtheorem{ques}[thm]{Question}

 \theoremstyle{definition}
\newtheorem{defn}[thm]{Definition}
\newtheorem{rem}[thm]{Remark}
\newtheorem{ex}[thm]{Example}
\newtheorem{notn}[thm]{Notation}

\newcommand{\bm}[1]{\mathbf{#1}}

\newcommand{\mb}[1]{\mathbb{#1}}
\newcommand{\mc}[1]{\mathcal{#1}}
\newcommand{\mf}[1]{\mathfrak{#1}}
\newcommand{\mr}[1]{\mathrm{#1}}

\newcommand{\vphi}{\varphi}

\newcommand{\pp}[1]{(\!({#1})\!)}
\newcommand{\bb}[1]{[\![{#1}]\!]}

\newcommand{\id}{\operatorname{id}}

\newcommand{\Char}{\operatorname{char}}

\newcommand{\Spec}{\operatorname{Spec}}

\newcommand{\Kmw}{\mathrm{K}^{\mathrm{MW}}}
\newcommand{\GW}{\operatorname{GW}}
\newcommand{\W}{\operatorname{W}}

\newcommand{\Tr}{\operatorname{Tr}}
\newcommand{\ind}{\operatorname{ind}}
\newcommand{\Vol}{\operatorname{Vol}}

\newcommand{\rank}{\operatorname{rank}}
\newcommand{\Jac}{\operatorname{Jac}}

\newcommand{\M}{\mathcal{M}_\circ}
\newcommand{\pcoor}[1]{%
  \begingroup\lccode`~=`: \lowercase{\endgroup
  \edef~}{\mathbin{\mathchar\the\mathcode`:}\nobreak}%
  [
  \begingroup
  \mathcode`:=\string"8000
  #1%
  \endgroup 
  ]
}

\makeatletter
\@namedef{subjclassname@2020}{\textup{2020} Mathematics Subject Classification}
\makeatother

\begin{document}
\title{Circles of Apollonius two ways}

\author{Stephen McKean}

\address{Department of Mathematics \\ Harvard University} 

\email{smckean@math.harvard.edu}
\urladdr{shmckean.github.io}

\subjclass[2020]{Primary: 14N15. Secondary: 14F52.}

\begin{abstract}
Because the problem of Apollonius is generally considered over the reals, it suffers from variance of number: there are at most eight circles simultaneously tangent to a given trio of circles, but some configurations have fewer than eight tangent circles. This issue arises over other non-closed fields as well. Using the tools of enriched enumerative geometry, we give two different ways to count the circles of Apollonius such that invariance of number holds over any field of characteristic not 2. We also pose the geometricity problem for local indices in enriched enumerative geometry.
\end{abstract}

\maketitle

\section{Introduction}
Given three general circles, there are eight circles that are tangent to all three. This classical theorem, known as Apollonius's problem or the circles of Apollonius, is in fact a corollary of B\'ezout's theorem. The moduli scheme of circles that are tangent to a given circle is a quadric surface in $\mb{P}^3$, and the circles of Apollonius correspond to the $2^3$ intersection points of three quadric surfaces.

All eight circles of Apollonius are only guaranteed to exist if one works over an algebraically closed field. However, the circles of Apollonius are generally studied over the reals --- perhaps because of the nice pictures that can be drawn in this setting. For real circles of Apollonius, a famous result is that anything but seven can happen: if one ignores multiplicity, there is a configuration of three real circles with $n$ real tangent circles for each $0\leq n<7$ and $n=8$~\cite{Ped70}. Even if one counts these tangent circles with multiplicity, there are configurations with fewer than eight real tangent circles (see Figure~\ref{fig:fewer than 8}). The loss of \textit{invariance of number} over non-closed fields is a common problem in enumerative geometry. The goal of this article is to restore invariance of number for the circles of Apollonius over any field of characteristic not 2. We will achieve this by giving a weighted, bilinear form-valued count of these circles.

\begin{figure}
\centering
\begin{tikzpicture}
    \draw[very thick] (0,1/8) circle (7/8);
    \draw[very thick] (5/4,0) circle (1);
    \draw[very thick] (1,2) circle (1/2);
	\draw[color=red, very thick] (.609,1.217) circle (.375);
	\draw[color=red, very thick] (.614,1.818) circle (.926);
	\draw[color=red, very thick] (.678,-.049) circle (1.575);
	\draw[color=red, very thick] (.777,.686) circle (1.833);
    \end{tikzpicture}
    \caption{Fewer than 8 circles over $\mb{R}$}\label{fig:fewer than 8}
\end{figure}
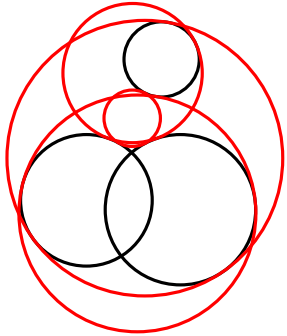

\begin{thm}\label{thm:main}
Let $k$ be a field of characteristic not 2. Let $C_1,C_2,C_3\subset\mb{P}^2_k$ be three circles whose centers do not lie on a shared line. Let $\mc{A}$ be the set of all circles that are tangent to all $C_i$. Finally, let $\mb{H}$ denote the isomorphism class of the hyperbolic bilinear form over $k$. Then each tangent circle $S\in\mc{A}$ determines an isomorphism class $\beta_S$ of bilinear forms such that
\begin{equation}\label{eq:main}
\sum_{S\in\mc{A}}\beta_S=4\mb{H}.
\end{equation}
\end{thm}

Taking the rank of Equation~\ref{eq:main} recovers the eight circles of Apollonius over $\overline{k}$. Other field invariants, such as signature, discriminant, and Hasse--Witt invariants, will tell us new results about the arithmetic enumerative geometry of the circles of Apollonius over fields like $\mb{R}$, $\mb{F}_q$, $\mb{Q}$, and so on.

As stated, Theorem~\ref{thm:main} is a direct corollary of the author's enrichment of B\'ezout's theorem~\cite{McK21}. The goal of this article is to give different geometric interpretations of the class $\beta_S$. The first interpretation is also a corollary of~\cite{McK21}.

\begin{thm}\label{thm:cone volume}
Assume the notation of Theorem~\ref{thm:main}. Each $C_i$ determines a quadric cone $Q_i\subset\mb{P}^3_k$. Each $S\in\mc{A}$ corresponds to an intersection point $s\in\bigcap_i Q_i$ with residue field $k(s)$. Let $\Vol(s):=\det(\nabla Q_i)|_s\in k(s)$ be the oriented volume of the parallelepiped spanned by the gradient vectors of the $Q_i$ at $s$. If no two of $C_1,C_2,C_3$ are tangent and $k(s)/k$ is separable, then $\beta_S=\Tr_{k(s)/k}\langle\Vol(s)\rangle$.
\end{thm}

Theorem~\ref{thm:cone volume} interprets $\beta_S$ as the \textit{intersection volume} of three cones at a point. While this interpretation is geometric, it is a step removed from the actual geometry of the circles of Apollonius. The following theorem gives a more intrinsic interpretation of $\beta_S$.

\begin{thm}\label{thm:alternating sum}
Assume the notation of Theorem~\ref{thm:main}. Assume that no two of $C_1,C_2,C_3$ are tangent. Let $S$ be a circle tangent to $C_1,C_2,C_3$, and assume that the field of definition $k(s)$ of $S$ is separable over $k$. Let $(a_i,b_i)$ be the center of $C_i$, $(a_s,b_s)$ the center of $S$, and $(x_i,y_i)$ the point at which $C_i$ and $S$ are tangent. Let $u_i=(a_i-x_i)(a_i-a_s)+(b_i-y_i)(b_i-b_s)$ and $v_i=(a_i-x_i)(x_i-a_s)+(b_i-y_i)(y_i-b_s)$. Finally, define
\begin{equation}\label{eq:area}
\mathrm{Area(s)}:=\sum_{\substack{\{i,m,n\}=\{1,2,3\}\\m<n}} (-1)^{i+1}u_iv_mv_n((a_m-a_s)(b_n-b_s)-(a_n-a_s)(b_m-b_s)).
\end{equation}
Then $\beta_S=\Tr_{k(s)/k}\langle\mathrm{Area(s)}\rangle$. In other words, $\beta_S$ can be interpreted as a weighted sum of the areas of the parallelograms determined by the centers of $S$ and $C_m,C_n$, where the weights record the ``direction of tangency'' of $S$ to each $C_i$.
\end{thm}

This article is part of the ongoing $\mb{A}^1$-enumerative geometry program (also known as \textit{quadratic} or \textit{enriched} enumerative geometry). Using tools from motivic homotopy theory, one is able to give bilinear form-valued answers to classical questions from enumerative geometry. The advantage of these bilinear form-valued counts is that one is no longer restricted to algebraically closed fields --- taking invariants of bilinear forms gives enumerative theorems over non-closed fields. See~\cite{Bra22,CDH20,DGGM21,KW21,Lev20,LP22,LV21,McK21,Pau20,SW21} for recent work in this area.

A central problem within $\mb{A}^1$-enumerative geometry is giving geometric interpretations for local indices. Classically (i.e. over algebraically closed fields), local indices are always interpreted as an intersection multiplicity. We ask whether local indices can always be viewed geometrically. We also ask whether one can classify enumerative problems in terms of the geometric description of their local indices.

\begin{ques}[See Question~\ref{ques:geometric interpretation}]\label{ques:geometricity}
Are local indices always geometric? Can enumerative problems be classified by the ``geometric taxon'' of their local indices?
\end{ques}

In Section~\ref{sec:bezout as universal}, we give a partial answer to Question~\ref{ques:geometricity} by showing that an analog of Theorem~\ref{thm:cone volume} holds for most enumerative problems. As with Theorem~\ref{thm:cone volume}, most of these analogous theorems will not be intrinsic to the enumerative problems at hand, so Theorem~\ref{thm:alternating sum} inspires us to look for a better answer to Question~\ref{ques:geometricity}. See Appendix~\ref{sec:phylogeny} for more on this \textit{geometricity} question.

\subsection{Outline}
In Section~\ref{sec:notation}, we set some relevant notation and give a brief overview of the $\mb{A}^1$-enumerative version of B\'ezout's theorem. In Section~\ref{sec:moduli spaces}, we discuss the parameter spaces of circles in the plane and circles tangent to a given circle. In Section~\ref{sec:euler}, we apply~\cite{McK21} to compute the $\GW(k)$-valued Euler number of the problem of Apollonius. We also treat variants of the problem where a subset of the original three circles are replaced with points.

Before continuing with the circles of Apollonius, we take a brief detour in Section~\ref{sec:local contributions}. We first give some context for Question~\ref{ques:geometricity}. We then exposit the dynamic local $\mb{A}^1$-degree of Pauli and Pauli--Wickelgren in Section~\ref{sec:dynamic degree} and give an alternative construction called the \textit{familial local degree} in Section~\ref{sec:familial}, both of which allow us to show that the intersection volume from B\'ezout's theorem gives an answer to Question~\ref{ques:geometricity}. However, the intersection volume describes the geometry of parameter spaces rather than the intrinsic geometry of the objects being counted, so we hope for a better answer to Question~\ref{ques:geometricity}. We speculate about what such an answer might look like in Appendix~\ref{sec:phylogeny}.

Returning to the circles of Apollonius, we prove Theorem~\ref{thm:alternating sum} in Section~\ref{sec:local contributions apollonius} with some supporting code in Appendix~\ref{sec:sage}. In Section~\ref{sec:other invariants}, we study symmetries of the set of circles of Apollonius that arise under inversion and degeneration. Using these symmetries, we describe a conjectural procedure (dependent on a few technical assumptions) for generating new geometric interpretations for the local indices in the problem of Apollonius. Supporting code for Section~\ref{sec:other invariants} is given in Appendix~\ref{sec:degen code}.

\subsection*{Acknowledgements}
We thank Marc Levine and Kirsten Wickelgren for helpful conversations and Sabrina Pauli for a correction. The author received support from an NSF MSPRF grant (DMS-2202825) and Kirsten Wickelgren's NSF CAREER grant (DMS-1552730).

\section{Notation and background}\label{sec:notation}
Throughout this article, we let $k$ be a field with $\Char{k}\neq 2$. We denote by $k\bb{t}$ and $k\pp{t}$ the ring of power series and the field of Laurent series over $k$, respectively.

Let $\mb{P}^n_k$ be projective $n$-space over $k$. We will be working with circles in the projective plane $\mb{P}^2_k$; we denote coordinates on this projective plane by $\pcoor{x:y:z}$. We will also work with the moduli space of circles in $\mb{P}^2_k$, which is isomorphic to $\mb{P}^3_k$; we will use the coordinates $\pcoor{c_0:c_1:c_2:c_3}$ when working with $\mb{P}^3_k$. We denote the projective variety cut out by homogeneous polynomials $f_1,\ldots,f_n$ by $\mb{V}(f_1,\ldots,f_n)$.

We denote by $\GW(k)$ the Grothendieck--Witt group of isomorphism classes of symmetric, non-degenerate bilinear forms over $k$. This group is generated by the elements $\langle a\rangle$ for $a\in k^\times$, which is the isomorphism class of the form $k\times k\to k$ defined by $(x,y)\mapsto axy$. The hyperbolic form will be denoted $\mb{H}:=\langle 1\rangle+\langle -1\rangle$. If $L$ is a finite separable extension of $k$, post-composition with the field trace determines a homomorphism $\Tr_{L/k}:\GW(L)\to\GW(k)$. 

In order to make use of $\Tr_{L/k}$, we will have a running assumption that $k(q)/k$ is a separable extension for any solution $q\in\mb{P}^3_k$ to the problem of Apollonius. This separability assumption is guaranteed if $k$ is perfect, if $[k(q):k]\leq 2$ (by our assumption that $\Char{k}\neq 2$), or if $\Char{k}>8$ (since $[k(q):k]\leq 8$ by the classical version of the circles of Apollonius).

We will frequently write $\ind_p\sigma$ when discussing local indices. Here, $\sigma$ refers to a section $\sigma:\mb{P}^3\to\mc{O}(2)^{\oplus 3}$ determined by a choice of three circles in the plane. More precisely, the space of circles tangent to a given circle is isomorphic to a quadric cone in $\mb{P}^3$, and $\sigma$ cuts out the three cones determined by our given trio of circles. The point $p$, which lies in the intersection of the three cones cut out by $\sigma$, corresponds to a circle tangent to our given three circles. The Nisnevich coordinates and local trivializations necessary to make sense of this local index are provided by the author's $\mb{A}^1$-enumerative treatment of B\'ezout's theorem~\cite{McK21}.

\subsection{B\'ezout's theorem}
Classically, B\'ezout's theorem counts the number of intersections (with multiplicity) of a collection of hypersurfaces in projective space. Over a non-closed field $k$, one also has to scale the intersection multiplicities by the degree of the residue field over $k$. However, this yields the same information as B\'ezout's theorem over the algebraic closure of $k$.

In order to develop a more interesting picture of B\'ezout's theorem over non-closed fields, we replace intersection multiplicity with \textit{intersection volume}. Let $f_1,\ldots,f_n$ be homogeneous polynomials in $k[x_0,\ldots,x_n]$. Given a common solution $p$ (i.e. $p\in\bigcap_i\mb{V}(f_i)$), write $\Vol(p):=\det(\nabla f_i)|_p$, which is the oriented volume of the parallelepiped spanned by the gradient vectors of $\mb{V}(f_i)$ at $p$. If the $\mb{V}(f_i)$ meet transversely at $p$, then $\Vol(p)\neq 0$ and hence we can take $\langle\Vol(p)\rangle\in\GW(k(p))$. This gives us~\cite[Theorem 1.2]{McK21}, which we restate below. We will demonstrate how to remove the transversality hypothesis in Section~\ref{sec:dynamic degree}.

\begin{thm}\label{thm:bezout}
Let $f_1,\ldots,f_n\in k[x_0,\ldots,x_n]$ be homogeneous of degrees $d_1,\ldots,d_n$ with $\sum_{i=1}^n d_i\equiv n+1\mod 2$. Assume that the $\mb{V}(f_i)$ meet transversely, and that $k(p)/k$ is separable for each $p\in\bigcap_i\mb{V}(f_i)$. Then
\[\sum_{p\in\bigcap_i\mb{V}(f_i)}\Tr_{k(p)/k}\langle\Vol(p)\rangle=\frac{d_1\cdots d_n}{2}\mb{H}.\]
\end{thm}

There are two sides to Theorem~\ref{thm:bezout}. The global count $\frac{d_1\cdots d_n}{2}\mb{H}$ comes from computing the Euler number of the bundle $\bigoplus_i\mc{O}(d_i)\to\mb{P}^n_k$; the results in Section~\ref{sec:euler} largely follow from this computation. The local contributions $\Tr_{k(p)/k}\langle\Vol(p)\rangle$ are computed using the work of Kass--Wickelgren~\cite{KW19,KW21} (see also~\cite{BBMMO21,BMP21}). While this computation is relatively straightforward, giving a geometric interpretation of these local contributions is the interesting step. 

The main results in the present article revolve around giving geometric interpretations of the local contributions in the context of the circles of Apollonius. These geometric interpretations will reveal a paradigm not present in enumerative geometry over closed fields: the problem of Apollonius is globally a special case of B\'ezout's theorem, but these enumerative problems are distinct from the local perspective.

\section{Moduli spaces of circles}\label{sec:moduli spaces}
We begin with a discussion of circles in algebraic geometry, following \cite[Section 2.3]{EH16}. A conic in the projective plane $\mb{P}^2_k$ is given by
\[\mb{V}(p_0x^2+p_1xy+p_2xz+p_3y^2+p_4yz+p_5z^2).\]
The moduli scheme of plane conics is thus isomorphic $\mb{P}^5_k$. A circle should be a conic of the form $(x-az)^2+(y-bz)^2-r^2z^2=0$ for some $a,b,r^2\in k$. Expanding this out, we have $x^2+y^2-2axz-2byz+(a^2+b^2-r^2)z^2=0$. This leads us to the following definition.

\begin{defn}
A \textit{circle} is a conic of the form
\[\mb{V}(p_0(x^2+y^2)+z(p_1x+p_2y+p_3z)).\]
Let $\M$ be the moduli space of circles in $\mb{P}^2_k$. Given $p=\pcoor{p_0:p_1:p_2:p_3}\in\mb{P}^3_k$, let
\[C(p)=\mb{V}(p_0(x^2+y^2)+z(p_1x+p_2y+p_3z))\in\M.\] 
If $p_0=0$, we say that $C({\pcoor{0:p_1:p_2:p_3}})$ is a \textit{degenerate circle}.
\end{defn}

The definition of $C$ gives us an explicit isomorphism $\mb{P}^3_k\cong\M$.

\begin{prop}
Regarded as a map, $C:\mb{P}^3_k\to\M$ is an isomorphism.
\end{prop}
\begin{proof}
Note that $C(p)$ does not depend on the choice of representative of $p$, so $C:\mb{P}^3_k\to\M$ is well-defined. The (well-defined) inverse morphism $C^{-1}:\M\to\mb{P}^3_k$ is given by $C^{-1}\mb{V}(p_0(x^2+y^2)+z(p_1x+p_2y+p_3z))=\pcoor{p_0:p_1:p_2:p_3}$. One can readily check that $C\circ C^{-1}=\id_{\mb{P}^3_k}$ and $C^{-1}\circ C=\id_{\M}$.
\end{proof}

\begin{rem}\label{rem:center and radius}
If $C(p)$ is a non-degenerate circle, then we can solve for the \textit{center} and \textit{radius squared} of $C(p)$ in terms of $p$. Since $p_0\neq 0$, we have
\begin{align*}
C(p)&=\mb{V}(p_0x^2+p_0y^2+p_1xz+p_2yz+p_3z^2)\\
&=\mb{V}((x+\tfrac{p_1}{2p_0}z)^2+(y+\tfrac{p_2}{2p_0}z)^2+(\tfrac{p_3}{p_0}-\tfrac{p_1^2}{4p_0^2}-\tfrac{p_2^2}{4p_0^2})z^2),
\end{align*}
which is a circle of radius squared $r^2:=-\tfrac{p_3}{p_0}+\tfrac{p_1^2}{4p_0^2}+\tfrac{p_2^2}{4p_0^2}$ with center $\pcoor{a:b:1}:=\pcoor{-\tfrac{p_1}{2p_0}:-\tfrac{p_2}{2p_0}:1}$. We will frequently write
\begin{align*}
\tfrac{p_1}{p_0}&=-2a,\\
\tfrac{p_2}{p_0}&=-2b,\\
\tfrac{p_3}{p_0}&=a^2+b^2-r^2.
\end{align*}
\end{rem}

\begin{rem}
When considering the set of circles tangent to a given trio of circles $C_1,C_2,C_3$, we will always assume that the centers of $C_1,C_2,C_3$ are not collinear. This is a generic condition, because there is a unique line through any pair of points.
\end{rem}

\begin{defn}
The \textit{residue field} or \textit{field of definition} of a circle $C(p)\in\M$ is the residue field $k(p)$ of the point $p\in\mb{P}^3_k$. If $C(p)$ is non-degenerate, then $k(p)/k$ is the minimal field extension such that $a,b,r^2\in k(p)$. Note in particular that $r$ need not be an element of $k(p)$.
\end{defn}

\subsection{The cone of tangent circles to a given circle}\label{sec:cone of tangent circles}
Given a non-degenerate circle $C(p)\in\M$, we would like to describe the space $Q(p)\subset\M$ of circles tangent to $C(p)$. By \cite[Section 2.3.2]{EH16}, $Q(p)$ is a quadric cone in $\M$ with cone point $C(p)$. We now describe a directrix for $Q(p)$, which allows us to explicitly solve for $Q(p)$ in terms of $p$.

\begin{prop}
Let $C(p)$ be a non-degenerate circle with radius squared $r^2$. Any circle of radius squared $(2r)^2$ with center on $C(p)$ is tangent to $C(p)$. (See Figure~\ref{fig:tangent circles}.)
\end{prop}
\begin{proof}
Let $\pcoor{a:b:1}$ be the center of $C(p)$. If $\pcoor{x_0:y_0:1}$ lies on the circle $C(p)$ (so that $(x_0-a)^2+(y_0-b)^2-r^2=0$), then $C(p)$ is tangent to $S:=\mb{V}((x-x_0z)^2+(y-y_0z)^2-(2r)^2z^2)$ at $q:=\pcoor{2a-x_0:2b-y_0:1}$. To verify that $q\in C(p)$ and $q\in S$, we simply check
\[(2a-x_0-a)^2+(2b-y_0-b)^2-r^2=(2a-2x_0)^2+(2b-2y_0)^2-(4r)^2=0.\]
To verify that $C(p)$ and $S$ are tangent at $q$, we compute the tangent spaces at $q$ using $T_q\mb{V}(f)=\mb{V}(\frac{\partial f}{\partial x}|_q\cdot x+\frac{\partial f}{\partial y}|_q\cdot y+\frac{\partial f}{\partial z}|_q\cdot z)$. Thus
\begin{align*}
T_qC(p)&=\mb{V}(2(a-x_0)\cdot x+2(b-y_0)\cdot y-2(a^2+b^2+r^2-ax_0-by_0)\cdot z),\\
T_qS&=\mb{V}(4(a-x_0)\cdot x+4(b-y_0)\cdot y-4(2r^2-x_0^2-y_0^2+ax_0+by_0)\cdot z).
\end{align*}
Substituting $2r^2=r^2+(x_0-a)^2+(y_0-b)^2$ in the defining equation for $T_qS$ shows that $T_qC(p)=T_qS$ as lines in $\mb{P}^2_k$.
\end{proof}

The family of circles of radius squared $(2r)^2$ with center on $C(p)$ will constitute our directrix for $Q(p)$.

\begin{prop}\label{prop:directrix}
Let $C(p)$ be a non-degenerate circle with center $\pcoor{a:b:1}$ and radius squared $r^2$. The family of circles of radius squared $(2r)^2$ with center on $C(p)$ is the circle
\begin{align}
\label{eq:directrix}
D:=C\mb{V}\big(&c_0((a^2+b^2+3r^2)c_0+ac_1+bc_2+c_3),\\
&c_1^2+c_2^2+4c_0((a^2+b^2-r^2)c_0+ac_1+bc_2)\big).\nonumber
\end{align}
\end{prop}
\begin{proof}
We obtain the defining equations for the family of circles of radius $(2r)^2$ with center on $C(p)$ by varying $\pcoor{x_0:y_0:1}\in C(p)$. Parametrically, we have $C(p)=\{\pcoor{a+r\tfrac{1-t^2}{1+t^2}:b+r\tfrac{2t}{1+t^2}:1}:t\in\mb{P}^1\}$. Let $E_t$ be the circle of radius squared $(2r)^2$ with center $\pcoor{a+r\tfrac{1-t^2}{1+t^2}:b+r\tfrac{2t}{1+t^2}:1}$. Then
\[E_t=C(\pcoor{1:-2(a+r\tfrac{1-t^2}{1+t^2}):-2(b+r\tfrac{2t}{1+t^2}):(a+r\tfrac{1-t^2}{1+t^2})^2+(b+r\tfrac{2t}{1+t^2})^2-4r^2}).\]
Let $\pcoor{c_0:c_1:c_2:c_3}$ be coordinates on $\mb{P}^3_k$. We then have the implicit description
\begin{align*}
\textstyle\bigcup_{t\in\mb{P}^1}E_t=C\mb{V}\big(&c_1^2+c_2^2-4c_0c_3-16r^2c_0^2,\\
&(c_1+2ac_0)^2+(c_2+2bc_0)^2-4r^2c_0^2\big)\\
=C\mb{V}\big(&c_1^2+c_2^2-4c_0c_3-16r^2c_0^2,\\
&c_1^2+c_2^2+4c_0((a^2+b^2-r^2)c_0+ac_1+bc_2)\big).
\end{align*}
Substituting $c_1^2+c_2^2=-4c_0((a^2+b^2-r^2)c_0+ac_1+bc_2)$, we find that $\bigcup_{t\in\mb{P}^1}E_t$ is given by Equation~\ref{eq:directrix}.
\end{proof}

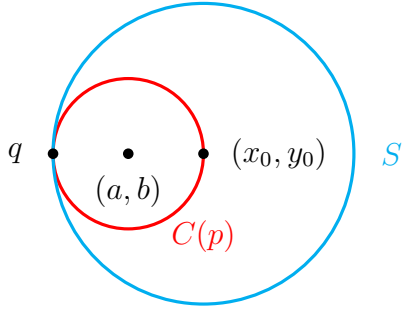
\begin{figure}
    \centering
    \begin{tikzpicture}
    \draw[color=red, very thick](0,0) circle (1);
    \node at (-1.5,0) {$q$};
    \fill (0,0) circle[radius=2pt];
    \node at (0,-1/2) {$(a,b)$};
    \draw[color=cyan, very thick](1,0) circle (2);
    \fill (1,0) circle[radius=2pt];
    \fill (-1,0) circle[radius=2pt];
    \node at (2,0) {$(x_0,y_0)$};
    \node[text=cyan] at (3.5,0) {$S$};
    \node[text=red] at (1,-1.1) {$C(p)$};
    \end{tikzpicture}
    \caption{Circle tangent to $C(p)$}
    \label{fig:tangent circles}
\end{figure}

Using the vertex $p$ and directrix from Equation~\ref{eq:directrix}, we now describe the cone $Q(p)$.

\begin{lem}\label{lem:cone}
Let $p=\pcoor{1:p_1:p_2:p_3}\in\mb{P}^3_k$. Let $\pcoor{a:b:1}$ and $r^2$ be the center and radius squared, respectively, of $C(p)$. Then 
\begin{align*}
Q(p)=C\mb{V}\left((aX+bY+Z)^2-r^2(X^2+Y^2)\right),
\end{align*}
where $X=c_1-p_1c_0$, $Y=c_2-p_2c_0$, and $Z=c_3-p_3c_0$.
\end{lem}
\begin{proof}
A cone in $\mb{P}^3_k$ with vertex $\pcoor{1:0:0:0}$ is given by the vanishing of $A_1c_1^2+A_2c_2^2+A_3c_3^2+A_4c_1c_3+A_5c_2c_3+A_6c_1c_2$ for some $A_1,\ldots,A_6$. In order to translate the vertex to $\pcoor{1:p_1:p_2:p_3}$, we replace $c_1$, $c_2$, and $c_3$ with $X$, $Y$, and $Z$, respectively. Next, we use the directrix for $Q(p)$ from Proposition~\ref{prop:directrix} to solve for $A_1,\ldots,A_6$. We will work in the open affine $\{c_0\neq 0\}\subset\mb{P}^3_k$, after which we will homogenize to obtain the desired equation for $Q(p)$.

On $\{c_0\neq 0\}$, Equation~\ref{eq:directrix} is defined by a circle on the hyperplane $\mb{V}((a^2+b^2+3r^2)c_0+ac_1+bc_2+c_3)$. This hyperplane allows us to set $Z|_D=-(a^2+b^2+3r^2+p_3)c_0-ac_1-bc_2$. Remark~\ref{rem:center and radius} implies that $p_1=-2a$, $p_2=-2b$, and $p_3=a^2+b^2-r^2$, so
\begin{align*}
Z|_D&=-2(a^2+b^2+r^2)c_0-ac_1-bc_2\\
&=-a(c_1+2ac_0)-b(c_2+2bc_0)-2r^2c_0\\
&=-aX-bY-2r^2c_0.
\end{align*}
We conclude by expanding $A_1X^2+A_2Y^2+A_3Z|_D^2+A_4XZ|_D+A_5YZ|_D+A_6XY$ and substituting $X=c_1+2ac_0$ and $Y=c_2+2bc_0$. Comparing to the coefficients of the directrix equation
\[c_1^2+c_2^2+4c_0((a^2+b^2-r^2)c_0+ac_1+bc_2)\]
allows us to solve for $A_1,\ldots,A_6$. We include some Sage code in Appendix~\ref{sec:sage} to perform the algebraic manipulations for us.
\end{proof}

Transversality is a generic condition, so a general choice of circles $C(p_1),C(p_2),C(p_3)$ will result in the cones $Q(p_1),Q(p_2),Q(p_3)$ meeting transversely. In fact, we can even characterize when these cones meet transversely in terms of our choice of circles.

\begin{prop}\label{prop:transverse cones}
The triple of cones $Q(p_1),Q(p_2),Q(p_3)$ meet transversely if and only if no two circles among $C(p_1),C(p_2),C(p_3)$ are tangent.
\end{prop}
\begin{proof}
By modifying the details of~\cite[Lemma 8.5]{EH16} to the case of circles, one can show that if a circle $C(q)$ has a point $x_i$ of simple tangency with $C(p_i)$, then the tangent plane $T_q Q(p_i)$ is the plane $H_i\subset\mb{P}^3_k$ of circles through $x_i$. Three 2-planes in $\mb{P}^3_k$ meet transversely at $q$ if and only if their intersection consists of a single point. That is, $Q(p_1),Q(p_2),Q(p_3)$ meet transversely if and only if $\bigcap_i H_i$ consists of a single point. 

Since $\M\cong\mb{P}^3_k$, there are finitely many circles through any three points in $\mb{P}^2_k$. In fact, there is a unique circle through any three points, as this coincides with the intersection of three 2-planes. (Note that if our three points are colinear, then their shared line is a circle of infinite radius.) In particular, $\bigcap_i H_i$ consists of a single point if and only if no circle can be tangent to two of $C(p_1),C(p_2),C(p_3)$ at a single point. We conclude by remarking that a circle can be tangent to two of $C(p_1),C(p_2),C(p_3)$ at a single point if and only if two of these circles are tangent.
\end{proof}

\subsection{The plane of circles through a point}\label{sec:circles through point}
Given a point $q=\pcoor{a:b:1}\in\mb{P}^2_k$ away from the line at infinity, we would like to describe the space $V(q)\subset\M$ of circles through $q$. In fact, any point in $\mb{P}^2_k$ determines an element of $V(q)$, so $V(q)$ is a hyperplane in $\M$.

\begin{lem}\label{lem:plane}
Let $q=\pcoor{a:b:1}\in\mb{P}^2_k$. Then
\[V(q)=C\mb{V}((a^2+b^2)c_0+ac_1+bc_2+c_3).\]
\end{lem}
\begin{proof}
The radius squared of any circle through $q$ is determined by its center. The circle through $q$ with center $\pcoor{A:B:1}$ and radius squared $r^2$ satisfies $(a-A)^2+(b-B)^2=r^2$, so this circle is given by
\begin{align*}
&\quad\ C(\pcoor{1:-2A:-2B:A^2+B^2-r^2})\\
&=C(\pcoor{1:-2A:-2B:A^2+B^2-(a-A)^2-(b-B)^2})\\
&=C(\pcoor{1:-2A:-2B:2Aa+2Bb-a^2-b^2}).
\end{align*}
The space of all such circles is defined implicitly by $C\mb{V}((a^2+b^2)c_0+ac_1+bc_2+c_3)$.
\end{proof}

\begin{rem}\label{rem:doubled}
A point $\pcoor{a:b:1}$ in $\mb{P}^2_k$ can be regarded as a circle with center $\pcoor{a:b:1}$ and radius squared 0. Under this perspective, the cone $Q(\pcoor{1:-2a:-2b:a^2+b^2})$ degenerates to the double plane $C\mb{V}((aX+bY+Z)^2)$, which is the plane $V(\pcoor{a:b:1})$ doubled.
\end{rem}

\section{Euler classes and relative orientability}\label{sec:euler}
In this section, we compute the fixed count of circles of Apollonius via the Euler class. There are several variants to the circles of Apollonius, because there are two ways in which a circle in $\mb{P}^2_k$ can differ from the non-degenerate circles we have considered thus far. First, a degenerate circle (i.e. a circle of the form $C(\pcoor{0:p_1:p_2:p_3})$) is a union of the line $\mb{V}(z)$ at infinity with another line in $\mb{P}^2_k$. Second, a non-degenerate circle with radius squared 0 is a point. One can thus ask how many circles are tangent to a given set of three objects, where each object may be a circle, line, or point. For simplicity, we will not consider any cases including lines (i.e. degenerate circles).

Each variant of the circles of Apollonius corresponds to studying the intersections of three hypersurfaces, each of the form $Q(p)$ or $V(p)$, in $\mb{P}^3_k$. The defining polynomials for $Q(p)$ and $V(p)$ described in Lemmata~\ref{lem:cone} and~\ref{lem:plane} will be used to determine a section $\sigma:\mb{P}^3_k\to\mc{O}(d_1)\oplus\mc{O}(d_2)\oplus\mc{O}(d_3)$, where each $d_i=1$ or 2. Each of these situations is a special case of B\'ezout's theorem \cite{McK21}. In this section, we will discuss the Euler class $e(\mc{O}(d_1)\oplus\mc{O}(d_2)\oplus\mc{O}(d_3))$ for each of these cases. In Section~\ref{sec:local contributions apollonius}, we will compute the local index $\ind_q\sigma$ \cite[Definition 30]{KW21} of our section at any tangent circle $C(q)$. This local index will give a new invariant on the circles of Apollonius.

\subsection{CCC}
Suppose we are given three general circles $C(p_1),C(p_2),C(p_3)\subset\mb{P}^2_k$. The set of circles tangent to these three circles are given by the intersection $Q(p_1)\cap Q(p_2)\cap Q(p_3)$. That is, we are intersecting three degree 2 hypersurfaces in $\mb{P}^3_k$. This is a special case of B\'ezout's theorem, with the defining equations of $Q(p_1),Q(p_2),Q(p_3)$ determining a section of $\mc{O}(2)^{\oplus 3}\to\mb{P}^3_k$.

\begin{prop}\label{prop:three circles}
The bundle $\mc{O}(2)^{\oplus 3}\to\mb{P}^3_k$ is relatively orientable with Euler class $4\mb{H}$.
\end{prop}
\begin{proof}
The relative orientability is given by \cite[Proposition 3.2]{McK21}, and the Euler class is computed in \cite[Theorem 4.4]{McK21}.
\end{proof}

\subsection{CCP}\label{sec:CCP}
Suppose we are given two general circles $C(p_1),C(p_2)\subset\mb{P}^2_k$ and a point $p_3\in\mb{P}^2_k$. If we consider $p_3$ as a circle of radius squared 0, then we can again use Proposition~\ref{prop:three circles} to check relative orientability and compute the Euler class. However, the local indices in this context fail to be interesting.

\begin{prop}\label{prop:CCP gives H}
Suppose $C(q_1)$ and $C(q_2)$ are non-degenerate circles with non-zero radius squared, and suppose $C(q_3)$ is a non-degenerate circle with radius squared 0. Suppose $Q(q_i)_{\mr{red}}$ intersect transversely at a point $q$ with $k(q)/k$ a separable extension. Then $\ind_q\sigma=\Tr_{k(q)/k}\mb{H}$.
\end{prop}
\begin{proof}
By \cite[Theorem 1.3]{BBMMO21}, we may assume that $q$ is $k$-rational. Since $C(q_1)$ and $C(q_2)$ are non-degenerate with non-zero radius squared, $Q(q_1)$ and $Q(q_2)$ are reduced. As discussed in Remark~\ref{rem:doubled}, $Q(q_3)$ is a double plane. By the transversality assumption on $Q(q_i)_{\mr{red}}$, it follows that the intersection multiplicity of the $Q(q_i)$ at $q$ is 2. By \cite[Proposition 5.2]{McK21}, it follows that $\rank(\ind_q\sigma)=2$, so \cite[Theorem 2]{QSW21} implies $\ind_q\sigma=\mb{H}$.
\end{proof}

The circles of Apollonius for two circles and a point only become interesting when we treat $p_3$ as a genuine point (rather than as a circle of radius squared 0). Circles tangent to $C(p_1)$ and $C(p_2)$ and through $p_3$ correspond to the intersection locus $Q(p_1)\cap Q(p_2)\cap V(p_3)$. This is B\'ezout's theorem for the bundle $\mc{O}(2)^{\oplus 2}\oplus\mc{O}(1)\to\mb{P}^3_k$. However, \cite[Proposition 3.2]{McK21} states that this bundle is not relatively orientable. One can relatively orient the bundle $\mc{O}(2)^{\oplus 2}\oplus\mc{O}(1)$ relative to the divisor of degenerate circles $\{c_0=0\}\subset\mb{P}^3_k$ (see \cite[Section 3.2]{McK21}), but the Euler class need not be independent of our choice of section. Nevertheless, we will still discuss the local indices for the cirlce-circle-point problem in Section~\ref{sec:local contributions apollonius}.

\subsection{CPP}
Suppose we are given a circle $C(p_1)\subset\mb{P}^2_k$ and two general points $p_2,p_3\in\mb{P}^2_k$. If we consider $p_2$ and $p_3$ as circles of radius squared 0, then we can again use Proposition~\ref{prop:three circles} to check relative orientability and compute the Euler class. However, the local indices in this context are again just hyperbolic forms.

\begin{prop}
Suppose $C(q_1)$ is a non-degenerate circle with non-zero radius squared, and suppose $C(q_2)$ and $C(q_3)$ are non-degenerate circles with radius squared 0. Suppose $Q(q_i)_{\mr{red}}$ intersect transversely at a point $q$ with $k(q)/k$ a separable extension. Then $\ind_q\sigma=\Tr_{k(q)/k}2\mb{H}$.
\end{prop}
\begin{proof}
By \cite[Theorem 1.3]{BBMMO21}, we may assume that $q$ is $k$-rational. Since $C(q_1)$ is non-degenerate with non-zero radius squared, $Q(q_1)$ is reduced. As discussed in Remark~\ref{rem:doubled}, $Q(q_2)$ and $Q(q_3)$ are double planes. By the transversality assumption on $Q(q_i)_{\mr{red}}$, it follows that the intersection multiplicity of the $Q(q_i)$ at $q$ is 4. By \cite[Proposition 5.2]{McK21}, it follows that $\rank(\ind_q\sigma)=4$.

Since $q$ is $k$-rational, we may change coordinates such that $q=\pcoor{1:0:0:0}$, the double plane $Q(q_2)$ is defined by $\mb{V}(\alpha c_1^2)$ for some $\alpha\in k^\times$, and the double plane $Q(q_3)$ is defined by $\mb{V}((\beta c_1+\gamma c_2)^2)$ for some $\beta\in k$ and $\gamma\in k^\times$. The cone $Q(q_1)$ is defined by $\mb{V}(F)$, where $F\in k[c_0,\ldots,c_3]$ is a degree 2 homogeneous polynomial satisfying $F(1,0,0,0)=0$. Let $f:=\frac{1}{c_0^2}F$. Using \cite{KW19}, we calculate $\ind_q\sigma$ by computing the EKL form on the local algebra
\[A:=\frac{k[c_1,c_2,c_3]_{(c_1,c_2,c_3)}}{(f,\alpha c_1^2,(\beta c_1+\gamma c_2)^2)}.\]
The rank of $\ind_q\sigma$ is equal to $\dim_k A$, so any four $k$-linearly independent elements of $A$ will form a $k$-basis. Since $\alpha$ and $\gamma$ are non-zero, it follows that $\{1,c_1,\beta c_1+\gamma c_2,c_1c_2\}$ is a $k$-basis of $A$. Let $E\in A$ be the distinguished socle element \cite[(4.7) Korollar]{SS75}, and let $\phi:A\to k$ be any $k$-linear form satisfying $\phi(E)=1$. Since $c_1^2=(\beta c_1+\gamma c_2)^2=0$ and
\begin{align*}
c_1c_2(\beta c_1+\gamma c_2)&=\gamma c_1c_2^2\\
&=\gamma^{-1} c_1(-\beta^2 c_1^2-2\beta\gamma c_1c_2)\\
&=0
\end{align*}
in $A$, the bilinear form $\Phi:A\times A\to k$ given by $\Phi(a,b)=\phi(ab)$ has the following presentation with respect to the basis $\{1,c_1,\beta c_1+\gamma c_2,c_1c_2\}$.
\begin{center}
    \begin{tabular}{r | c c c c}
    & $1$ & $c_1$ & $\beta c_1+\gamma c_2$ & $c_1c_2$ \\
    \hline
    $1$ & $*$ & $*$ & $*$ & $\phi(c_1c_2)$\\
    $c_1$ & $*$ & 0 & $\gamma\cdot\phi(c_1c_2)$ & 0\\
    $\beta c_1+\gamma c_2$ & $*$ & $\gamma\cdot\phi(c_1c_2)$ & 0 & 0\\
    $c_1c_2$ & $\phi(c_1c_2)$ &0  & 0 & 0
    \end{tabular}
\end{center}
The bilinear form $\Phi$ is non-degenerate by \cite[Lemma 6]{KW19}, so $\Phi=2\mb{H}$ in $\GW(k)$.
\end{proof}

We will thus treat $p_2$ and $p_3$ as genuine points (rather than as circles of radius squared 0). Circles tangent to $C(p_1)$ and through $p_2,p_3$ correspond to the intersection locus $Q(p_1)\cap V(p_2)\cap V(p_3)$. This is B\'ezout's theorem for the bundle $\mc{O}(2)\oplus\mc{O}(1)^{\oplus 2}\to\mb{P}^3_k$. 

\begin{prop}
The bundle $\mc{O}(2)\oplus\mc{O}(1)^{\oplus 2}\to\mb{P}^3_k$ is relatively orientable with Euler class $\mb{H}$.
\end{prop}
\begin{proof}
The relative orientability and Euler class computation can be found in \cite[Proposition 3.2 and Theorem 4.4]{McK21}.
\end{proof}

\subsection{PPP}
Finally, suppose we are given three general points $p_1,p_2,p_3\in\mb{P}^2_k$. If we consider these points as circles of radius squared 0, then Proposition 4.1 again gives us relative orientability and computes the relevant Euler class. However, the intersection of three general double planes in $\mb{P}^3_k$ will consist of a single point, so the local index will be equal to the Euler class:
\[\ind_q\sigma=e(\mc{O}(2)^{\oplus 3})=4\mb{H}.\]
As in the previous cases involving points instead of circles, we will treat $p_1,p_2,p_3$ as genuine points. The unique circle through $p_1,p_2,p_3$ corresponds to the intersection $V(p_1)\cap V(p_2)\cap V(p_3)$. This is B\'ezout's theorem for the bundle $\mc{O}(1)^{\oplus 3}\to\mb{P}^3_k$. As for the circle-circle-point problem, this bundle is not relatively orientable by \cite[Proposition 3.2]{McK21}. One can relatively orient $\mc{O}(1)^{\oplus 3}$ relative to the divisor of degenerate circles $\{c_0=0\}\subset\mb{P}^3_k$, but the Euler class is equal to the local index and will depend on the choice of section.

\section{Local contributions for general enumerative problems}\label{sec:local contributions}
Results in enumerative geometry often consist of equations relating a fixed count of objects to a sum of \textit{local contributions} that depend on the individual objects being counted:
\begin{align}\label{eq:enumerative}
\text{fixed count}=\sum_\text{objects}\text{local contribution}.
\end{align}
For example, in the classical version of the circles of Apollonius, the fixed count is 8, and each tangent circle gives a local contribution of 1. In $\mb{A}^1$-enumerative geometry, both the fixed count and local contributions are $\GW(k)$-valued rather than integer-valued. In many cases, fixed counts can be computed using a motivic version of the Euler class \cite{Lev20,KW21,BW20}. Local contributions are computed as a local index, which admits a convenient formula in terms of commutative algebra \cite{KW19,BBMMO21,BMP21}.

In order for Equation~\ref{eq:enumerative} to be an enumerative \textit{geometric} equation, we need to give geometric descriptions of the local contributions. Giving a meaningful geometric interpretation of the local index, which is a priori an algebraic expression, poses one of the main difficulties in $\mb{A}^1$-enumerative geometry.

\begin{ques}[Geometricity]\label{ques:geometric interpretation}
Are local indices always geometric? Can enumerative problems be classified by the ``geometric taxon'' of their local indices?
\end{ques}

In a sense to be described in Section~\ref{sec:bezout as universal}, B\'ezout's theorem gives a universal geometric interpretation of local contributions. However, this perspective fails to give a satisfactory answer to Question~\ref{ques:geometric interpretation} --- B\'ezout's theorem gives a geometric interpretation is in terms of the moduli space of the geometric objects in question, rather than in terms of the intrinsic geometry of the objects themselves. We will clarify this concern in Section~\ref{sec:local contributions apollonius} with the circles of Apollonius as a case study. While the classical statement of the circles of Apollonius can be viewed as a corollary of B\'ezout's theorem, the geometric interpretation given in Lemma~\ref{lem:geometric interpretation} shows that the $\mb{A}^1$-enumerative situation is more subtle. 

\begin{rem}
In light of the previous paragraph, we refine Question~\ref{ques:geometric interpretation} to ask whether local indices are ``intrinsically'' geometric, as demonstrated in the following example.
\end{rem}

\begin{ex}\label{ex:taxa}
The second part of Question~\ref{ques:geometric interpretation} asks for a taxonomy of enumerative problems in terms of the geometric interpretations of their local indices. We propose three potential taxa to give an indication of what this might look like. See Appendix~\ref{sec:phylogeny} for further discussion on how such taxa might fit together into a phylogenetic tree of enriched enumerative problems.
\begin{enumerate}[(i)]
\item \textit{Segre involutions} play a prominent role in Kass--Wickelgren's enriched count of lines on cubic surfaces~\cite{KW21} and Pauli's count of lines on quintic threefolds~\cite{Pau20}. The Segre involution associated to a line $L$ on a cubic surface $X$ swaps points $p,q\in L$ such that $T_pX=T_qX$. Kass and Wickelgren show that the local index for lines on cubic surfaces is given by the degree of the Segre involution. The description of Segre involutions associated to lines on quintic threefolds is a little more complicated, but it again relates to swapping points along a line whose tangent spaces coincide. Pauli shows that the local index for lines on quintic threefolds is given by the degree of a product of three Segre involutions. In general, one might hope that the local index for counting lines on a degree $2n-3$ hypersurface in $\mb{P}^n$ can be described in terms of Segre involutions.

\item There are 2 lines meeting 4 lines in $\mb{P}^3$, and in general a finite number of lines meeting $2n-2$ hyperplanes of dimension $n-2$ in $\mb{P}^n$. Srinivasan and Wickelgren give an $\mb{A}^1$-enumerative account of this story when $n$ is odd~\cite{SW21}. For $n=3$, the local index is a difference of cross-ratios associated to the geometry of the solution lines $L,L'$, the given lines $L_1,L_2,L_3,L_4$, their various intersections, and the various planes spanned by pairs of intersecting lines. For larger $n$, Srinivasan and Wickelgren geometrically interpret the local index by a determinantal formula that again depends on the various intersection points and hyperplanes spanned by pairs of intersecting hyperplanes. This (more complicated) interpretation recovers the difference of cross-ratios when $n=3$, so this family of enumerative problems share a common geometric local index.

\item In~\cite{DGGM21}, the authors give an enriched count of conics meeting 8 lines in $\mb{P}^3$. Given a conic $C$ meeting lines $L_1,\ldots,L_8$, the local index is geometrically described in terms of the intersection points $C\cap L_i$ and the slopes of each $L_i$ relative to to the tangent lines $T_{C\cap L_i}C$. The local index comes from an explicit section of the bundle $\mc{O}(1)^{\oplus 8}\to\mb{P}\mr{Sym}^2(\mc{S}^\vee)$, where $\mc{S}\to\mb{G}(2,3)$ is the tautological subbundle on the Grassmannian of 2-planes in $\mb{P}^3$.

In general, there are a finite number of degree $n$ plane curves meeting $f(n):=\binom{n+2}{n}+2$ lines in $\mb{P}^3$. The bundle $\mc{O}(1)^{\oplus f(n)}\to\mb{P}\mr{Sym}^n(\mc{S}^\vee)$ is relatively orientable if and only if $\binom{n+2}{n}$ is even (see~\cite[Lemma 3.1]{DGGM21}), which happens precisely when $n$ is equivalent to 2 or 3 mod 4. In any case, we again get an explicit section whose associated local index can be described geometrically in terms of the intersection points $C\cap L_i$ and the slopes of $L_i$ relative to $T_{C\cap L_i}C$ (where $C$ is now a plane curve of degree $n$).

From this perspective, the enumerative problems of counting plane curves of a given degree meeting lines in $\mb{P}^3$ belong to the same geometric taxon. An interesting question is whether the problems of counting $\mb{P}^m$-curves of a given degree meeting lines in $\mb{P}^n$ (with $m\leq n$) also belong to this geometric taxon.
\end{enumerate}
\end{ex}

\subsection{Intersection volume as a universal local contribution}\label{sec:bezout as universal}
In classical enumerative geometry, many theorems only become truly \textit{enumerative} when objects are counted with multiplicity. B\'ezout's theorem is the prototypical example of this phenomenon: unless intersections are counted with multiplicity, the product of degrees merely gives an upper bound to the number of intersections. In this way, intersection multiplicity is a universal local contribution for many classical enumerative problems. Similarly, B\'ezout's theorem gives a universal geometric interpretation of local contributions in $\mb{A}^1$-enumerative geometry. We will first describe this geometric interpretation under a transversality assumption. We will then use Pauli's enrichment of the dynamic degree \cite{Pau20,PW20} to reduce to the transverse case.

\begin{defn}
Let $f_1,\ldots,f_n\in k[x_1,\ldots,x_n]$ with corresponding hypersurfaces $X_i:=\mb{V}(f_i)\subseteq\mb{A}^n_k$. Assume that $p\in\bigcap_i X_i$ is an isolated intersection point with intersection multiplicity $i_p(X_1,\ldots,X_n)=[k(p):k]$ (that is, $X_i$ intersect transversely at $p$ by \cite[Proposition 5.4]{McK21}). The \textit{intersection volume} $\Vol(p)\in k(p)$ of $f_1,\ldots,f_n$ at $p$ is the volume of the parallelepiped spanned by the gradient vectors $\nabla f_i(p)$. In other terms,
\begin{align*}
\Vol(p)&=\det(\nabla f_1(p)\mid \ldots\mid \nabla f_n(p))\\
&=\Jac(f_1,\ldots,f_n)(p).
\end{align*}
\end{defn}

In order to compute the intersection volume for a section $\sigma:X\to V$ of a relatively oriented vector bundle, we need Nisnevich coordinates and compatible local trivializations to express $\sigma$ as an endomorphism of affine space. This intersection volume will not depend on our choices of such data \cite[Corollary 31]{KW21}, but we need to show that such data are guaranteed to exist. The existence of Nisnevich coordinates is given by \cite[Lemma 19]{KW21}. We now show that compatible local trivializations also exist.

\begin{prop}\label{prop:trivs exist}
Let $X$ be a $k$-scheme of dimension $n$. Let $V\to X$ be a relatively orientable vector bundle of rank $n$, and let $\rho:\det{V}\otimes\omega_X\xrightarrow{\cong}L^{\otimes 2}$ be a relative orientation of $V\to X$. Given Nisnevich coordinates $\vphi:U\to\mb{A}^n_k$ on an open subscheme $U\subseteq X$, there exists a local trivialization $\psi:V|_U\to\mb{A}^n_k\times U\to\mb{A}^n_k\to\mb{A}^n_k$ that is compatible with $(\vphi,U)$ and $(\rho,L)$.
\end{prop}
\begin{proof}
The Nisnevich coordinates $\vphi=(\vphi_1,\ldots,\vphi_n)$ on $U$ determine a local trivialization $d\vphi:=(d\vphi_1,\ldots,d\vphi_n)$ on the cotangent bundle $T^*X|_U$, which in turn determines the distinguished basis element $\det{d\vphi}:=d\vphi_1\wedge\cdots\wedge d\vphi_n$ of $\omega_X|_U:=\det T^*X|_U$ (considered as a rank one $\mc{O}_X(U)$-module). Let $\psi=(\psi_1,\ldots,\psi_n):V|_U\to\mb{A}^n_k\times U\to\mb{A}^n_k$ be a local trivialization. This determines the distinguished basis element $\det\psi:=\psi_1\wedge\cdots\wedge\psi_n$ of $\det{V}|_U$ (considered as a rank one $\mc{O}_X(U)$-module).

If $\rho(\det\psi\otimes\det d\vphi)=\ell\otimes\ell$ for some $\ell\in L|_U$, then we are done. Otherwise, note that $L|_U^{\otimes 2}\cong L|_U\cong\mc{O}_X(U)$ and hence $\rho(\det\psi\otimes\det d\vphi)\in L|_U^{\otimes 2}\cong\mc{O}_X(U)$. Let $f\in\mc{O}_X(U)$ be the image of $\rho(\det\psi\otimes\det d\phi)$, and let $\psi'=(f\psi_1,\psi_2,\ldots,\psi_n)$. Now the image in $\mc{O}_X(U)$ of $\rho(\det\psi'\otimes\det d\phi)$ is $f^2$. Letting $\ell\in L|_U$ be the preimage of $f$ under $L|_U\cong\mc{O}_X(U)$, we have $\rho(\det\psi'\otimes\det d\vphi)=\ell\otimes\ell$.
\end{proof}

We can now show that the local index of a section $\sigma:X\to V$ at a simple zero is always given by an intersection volume.

\begin{prop}\label{prop:transverse}
Let $X$ be a $k$-scheme of dimension $n$. Let $V\to X$ be a relatively orientable vector bundle of rank $n$, and let $\rho:\det{V}\otimes\omega_X\xrightarrow{\cong}L^{\otimes 2}$ be a relative orientation of $V\to X$. Let $\sigma:X\to V$ be a section. If $p\in\sigma^{-1}(0)$ is a simple zero with separable residue field $k(p)/k$, then the local index $\ind_p\sigma$ is equal to the intersection volume $\Tr_{k(p)/k}\langle\Vol(p)\rangle$.
\end{prop}
\begin{proof}
This is essentially proved in \cite[Lemma 5.5]{McK21}, but we repeat the relevant details here. Let $U\subset X$ be an open neighborhood of $p$ with Nisnevich coordinates $\vphi:U\to\mb{A}^n_k$, which exist by \cite[Lemma 19]{KW21}. Let $\psi:V|_U\to\mb{A}^n_k\times U\to\mb{A}^n_k$ be a local trivialization of $V$ compatible with the Nisnevich coordinates $\vphi$ and the relative orientation $(\rho,L)$, which exists by Proposition~\ref{prop:trivs exist}. Then there exist $f_1,\ldots,f_n\in k[x_1,\ldots,x_n]$ such that
\[(f_1,\ldots,f_n)=\psi\circ\sigma\circ\vphi^{-1}:\mb{A}^n_k\to\mb{A}^n_k.\]
By \cite[Corollary 31]{KW21}, the local index $\ind_p\sigma$ is well-defined and independent of our choice of Nisnevich coordinates, compatible trivializations, and functions $f_1,\ldots,f_n$. By \cite[Proposition 34]{KW21} (see also \cite{BBMMO21}), we have $\ind_p\sigma=\Tr_{k(p)/k}\ind_{\tilde{p}}\sigma_{k(p)}$, where $\tilde{p}$ is the $k(p)$-rational lift of $p$ determined by the extension $k\to k(p)$, and $\sigma_{k(p)}$ is the base change of $\sigma$. Since $p$ is a simple zero of $\sigma$, we have $\ind_{\tilde{p}}\sigma_{k(p)}=\langle\Jac(f_1,\ldots,f_n)(p)\rangle$ by \cite[Proposition 15]{KW19}, which is equal to $\langle\Vol(p)\rangle$ by \cite[Section 5.1]{McK21}.
\end{proof}

\begin{rem}
An instance of Proposition~\ref{prop:transverse} can be seen in Brazelton's enriched count of $m$-planes meeting $(n-m)$-planes in $\mb{P}^n$~\cite{Bra22}. One geometric interpretation given in the article is phrased as a signed or oriented volume in terms of certain vectors represented in Pl\"ucker coordinates. These Pl\"ucker coordinates are local coordinates on the Grassmannian $X$ parameterizing the planes under consideration, and the oriented volume is the intersection volume of hypersurfaces that are locally determined by a relevant section of a vector bundle on $X$.
\end{rem}

\subsection{Dynamic local $\mb{A}^1$-degree}\label{sec:dynamic degree}
Since transversality is a generic condition, we would like to reduce arbitrary intersections to the case of Proposition~\ref{prop:transverse}. Using dynamic intersections, one can relate a special intersection to a generic one \cite[Section 11]{Ful98}. The classical dynamic intersection was enriched by Pauli to give a dynamic $\mb{A}^1$-Euler number \cite{Pau20}, as well as by Pauli--Wickelgren to give a dynamic $\mb{A}^1$-Milnor number \cite{PW20}. In fact, Pauli--Wickelgren's approach can be repeated almost verbatim to give a dynamic local $\mb{A}^1$-degree for any $f:\mb{A}^n_k\to\mb{A}^n_k$. We recall the details here, with \cite[Section 6.3]{PW20} as the standing reference for Section~\ref{sec:dynamic degree}. This will culminate in Theorem~\ref{thm:dynamic degree}, which is essentially a rephrasing of \cite[Theorem 5]{PW20}. We begin with the classical computation of $\GW(k\bb{t})$:

\begin{prop}\label{prop:GW(k[[t]])}
The map $\langle r(t)\rangle\mapsto \langle r(0)\rangle$ defines an isomorphism $\mr{ev}_0:\GW(k\bb{t})\xrightarrow{\cong}\GW(k)$ with inverse induced by the inclusion map $k\hookrightarrow k\bb{t}$.
\end{prop}
\begin{proof}
This follows from the fact that $\GW(k\bb{t})$ is generated by elements of the form $\langle r(t)\rangle$ for units $r(t)\in k\bb{t}^\times$ (see e.g. \cite[Lemma B.3]{BW20}). Any such unit satisfies $r(0)\neq 0$, so we may write $r(t)=\sum_{i=0}^\infty a_it^i=a_0(1+\sum_{i=1}^\infty\frac{a_i}{a_0}t^i)$. Since $\Char{k}\neq 2$, there exists a square root $s(t)\in k\bb{t}$ of $1+\sum_{i=1}^\infty\frac{a_i}{a_0}t^i$. In particular, in $\GW(k\bb{t})$, we have
\begin{align*}
\langle r(t)\rangle&=\langle a_0s(t)^2\rangle\\
&=\langle a_0\rangle=\langle r(0)\rangle.\qedhere
\end{align*}
\end{proof}

We now summarize the relationships between $\GW(k)$, $\GW(k\bb{t})$, and $\GW(k\pp{t})$. Any element of $k\pp{t}$ is either of the form $u$ or $ut$, where $u$ is a unit under the $(t)$-adic valuation. The second residue homomorphism $\partial_t:\GW(k\pp{t})\to\W(k)$ is defined by $\partial_t\langle ut\rangle=\langle\bar{u}\rangle$ and $\partial_t\langle u\rangle=0$, where $\bar{u}\in k$ is the residue of $u$ in $k\pp{t}/(t)$. More generally, the second residue homomorphism is defined on Milnor--Witt $K$-groups
\[\partial_t:\Kmw_n(k\pp{t})\to \Kmw_{n-1}(k),\]
with $\Kmw_n(k\bb{t}):=\ker\partial_t$ \cite[p. 58]{Mor12}. Setting $n=0$ recovers the short exact sequence
\[0\to\GW(k\bb{t})\to\GW(k\pp{t})\xrightarrow{\partial_t}\W(k)\to 0\]
of abelian groups. Denote the inclusion of $\ker\partial_t$ by $\imath:\GW(k\bb{t})\to\GW(k\pp{t})$. Let $\jmath:\GW(k)\hookrightarrow\GW(k\pp{t})$ be inclusion into the first factor under Springer's theorem \cite[Chapter VI, Theorem 1.4]{Lam05}. Together with $\mr{ev}_0:\GW(k\bb{t})\to\GW(k)$, these maps form a commutative triangle.

\begin{prop}\label{prop:triangle of gw}
The following diagram commutes.
\[\begin{tikzcd}
\GW(k\bb{t})\arrow[r,hook,"\imath"]\arrow[d,"\mr{ev}_0"',"\cong"] & \GW(k\pp{t})\\
\GW(k)\arrow[ur,hook,"\jmath"']&
\end{tikzcd}\]
\end{prop}
\begin{proof}
The composite $\imath\circ\mr{ev}_0^{-1}:\GW(k)\hookrightarrow\GW(k\pp{t})$ is the injection induced by the inclusion $k\hookrightarrow k\pp{t}$ defined by $a\mapsto a$ \cite[p. 146]{Lam05}, which is precisely the inclusion $\jmath$.
\end{proof}

Given a map $\mb{A}^n_k\to\mb{A}^n_k$, we will build a deformation (that is, a map over $k\bb{t}$) whose local degree at a special fiber is our local degree valued in $\GW(k)$. The general fiber of this map will have local degree valued in $\GW(k\pp{t})$. Proposition~\ref{prop:triangle of gw} will enable us to relate these two local degrees and exploit the genericity of transverse intersections.

Let $f_1,\ldots,f_n\in k[x_1,\ldots,x_n]$. Given $g_1,\ldots,g_n\in k\bb{t}[x_1,\ldots,x_n]$, let
\[X:=\mb{V}(f_1+tg_1,\ldots,f_n+tg_n)\subseteq\mb{A}^n_{k\bb{t}}.\]

\begin{notn}
Given a scheme $Y\to\Spec{k\bb{t}}$, denote its special fiber by $Y_0:=\Spec{k}\times_{\Spec{k\bb{t}}}Y$ and its generic fiber by $Y_t:=\Spec{k\pp{t}}\times_{\Spec{k\bb{t}}}Y$.
\end{notn}

Note that $\mb{V}(f_1,\ldots,f_n)=X_0$. Since $k\bb{t}$ is a local ring, \cite[\href{https://stacks.math.columbia.edu/tag/04GG}{Lemma 04GG (12)}]{stacks} implies that $X=X^{\mr{fin}}\amalg X^{\geq 1}$, where $X^{\mr{fin}}\to\Spec{k\bb{t}}$ is finite and $(X^{\geq 1})_0$ is a union of irreducible $k$-schemes of dimension at least 1.

\begin{notn}
Given a closed point $p\in X_0$, let $X^p$ be the union of all irreducible components of $X$ containing $p$.
\end{notn}

The scheme $X^p$ is a finite collection of points, namely $p$ and the points that $p$ splits into in the generic fiber $X_t$. To see this, we first show that $X^p\to\Spec{k\bb{t}}$ is finite.

\begin{prop}\label{prop:X^p finite}
If $p\in X_0$ is isolated, then $X^p\to\Spec{k\bb{t}}$ is finite.
\end{prop}
\begin{proof}
Since $p$ is isolated, the local ring $\mc{O}_{X_0,p}$ is finite as a $k$-module. In particular, the special fiber of any irreducible component of $X$ containing $p$ must be finite over $k$, so the decomposition $X=X^{\mr{fin}}\amalg X^{\geq 1}$ implies that $X^p$ is a closed subscheme of $X^{\mr{fin}}$. The finiteness of $X^p\to\Spec{k\bb{t}}$ now follows from the finiteness of $X^{\mr{fin}}\to\Spec{k\bb{t}}$.
\end{proof}

We are now ready to show that $p$ is the only point in the special fiber $(X^p)_0$. Since $Y=Y_0\amalg Y_t$ (set-theoretically) for any $k\bb{t}$-scheme $Y$, it will follow that $X^p-\{p\}=(X^p)_t$. By construction, $(X^p)_t$ consists of the points that map to $p$ under $X_t\to X_0$, or in other words, the points that $p\in X_0$ splits into in the generic fiber $X_t$.

\begin{prop}\label{prop:X^p_0=p}
If $p\in X_0$ is isolated, then $p$ is the only point of $(X^p)_0$.
\end{prop}
\begin{proof}
Let $x_t\in(X^p)_t$ be a point. The residue field $\kappa(x_t)$ of $x_t$ is a finite extension of $k\pp{t}$. Letting $R$ be the integral closure of $k\bb{t}$ in $\kappa(x_t)$, we get a commutative diagram
\begin{equation}\label{eq:diagram}
\begin{tikzcd}
\Spec{\kappa(x_t)}\arrow[r]\arrow[d] & X^p\arrow[d]\\
\Spec{R}\arrow[r] & \Spec{k\bb{t}}.
\end{tikzcd}
\end{equation}
The map $\Spec{R}\to\Spec{k\bb{t}}$ is finite by \cite[\href{https://stacks.math.columbia.edu/tag/032Q}{Lemma 032Q}]{stacks} and \cite[\href{https://stacks.math.columbia.edu/tag/032L}{Lemma 032L}]{stacks} if $\Char{k}=0$ or \cite[\href{https://stacks.math.columbia.edu/tag/032N}{Lemma 032N}]{stacks} if $\Char{k}\neq 0$. Since $X^p\to\Spec{k\bb{t}}$ is finite, it is also a proper morphism, so the valuative criterion for properness \cite[\href{https://stacks.math.columbia.edu/tag/0A40}{Lemma 0A40}]{stacks} implies that there is a unique morphism $\Spec{R}\to X^p$ that commutes with Diagram~\ref{eq:diagram}. Moreover, the image of $\Spec{R}$ (which we denote by $x$) is a component of $X^p$, since having a finite map $X^p\to\Spec{k\bb{t}}$ implies that $\dim{X^p}\leq\dim{\Spec{k\bb{t}}}=\dim{R}$ \cite[\href{https://stacks.math.columbia.edu/tag/0ECG}{Lemma 0ECG}]{stacks}. Thus $p\in x$ by definition of $X^p$. 

By \cite[Section 3.2.4, Theorem 2]{BGR84}, $R$ is a complete discrete valuation ring. Since $k\bb{t}\to R$ is finite, $\dim{R}=1$. The Cohen structure theorem (see e.g. \cite[\href{https://stacks.math.columbia.edu/tag/0323}{Section 0323}]{stacks}) thus implies that $R\cong L\bb{u}$ for some finite extension $L/k$ and some parameter $u$. In particular, the special fiber $(\Spec{R})_0$ contains a unique point, so the special fiber $x_0\in (X^p)_0$ consists of a unique point. As $p\in x$, it follows that $x_0=p$, so the special fiber of any component of $X^p$ consists solely of the point $p$.
\end{proof}

Our next goal is to show that $\mc{O}_{X^p}(X^p)$ is a free $k\bb{t}$-module and that $(f_1+tg_1,\ldots,f_n+tg_n)$ is  regular sequence. This will allow us to define the local degree $\deg_p^{\mb{A}^1}(f_1+tg_1,\ldots,f_n+tg_n)\in\GW(k\bb{t})$ as the isomorphism class of the Scheja--Storch bilinear form $\mc{O}_{X^p}(X^p)\times\mc{O}_{X^p}(X^p)\to k\bb{t}$~\cite[\S 3]{SS75}.

\begin{prop}\label{prop:X^p local}
If $p\in X_0$ is isolated, then $\mc{O}_{X^p}(X^p)$ is a local ring.
\end{prop}
\begin{proof}
In the proof of Proposition~\ref{prop:X^p_0=p}, we saw that (as a set of points) $X^p$ consists of a set of maximal points $x_t\in(X^p)_t$ and a unique closed point $p\in(X^p)_0$. Since $X^p\to\Spec{k\bb{t}}$ is finite, it is also quasi-compact. In particular, the topological space underlying $X^p$ is quasi-compact by~\cite[\href{https://stacks.math.columbia.edu/tag/01K4}{Lemma 01K4}]{stacks}. Now~\cite[Proposition 4]{Che17} implies that $X^p=\Spec{R}$ for some local ring $R$. In particular, $X^p$ is affine, so $R\cong\mc{O}_{X^p}(X^p)$.
\end{proof}

\begin{prop}\label{prop:X^p reg}
If $p\in X_0$ is isolated, then there exists a $k\bb{t}$-module $M$ such that $(f_1+tg_1,\ldots,f_n+tg_n)$ is a regular sequence in $M$ and $\mc{O}_{X^p}(X^p)\cong M/(f_1+tg_1,\ldots,f_n+tg_n)$.
\end{prop}
\begin{proof}
Let $\mf{m}\subset k\bb{t}[x_1,\ldots,x_n]$ be the maximal ideal corresponding to the point $p$ over $k\bb{t}$ (with $t\in\mf{m}$), and let $\mf{m}_0=\mf{m}/(t)=\mf{m}\cap k[x_1,\ldots,x_n]$ (which corresponds to $p$ over $k$). Set $R=\frac{k\bb{t}[x_1,\ldots,x_n]}{(f_1+tg_1,\ldots,f_n+tg_n)}$. Let $\min(R)$ be the set of minimal primes of $R$ (corresponding to the irreducible components of $X$), and let $\min(R)^p$ be the set of minimal primes of $R$ that are contained in the image of $\mf{m}$ (corresponding to the irreducible components of $X^p$). Finally, let $S=R-\min(R)^p$. We claim that $\mc{O}_{X^p}(X^p)\cong S^{-1}R$, from which it will follow that $\mc{O}_{X^p}(X^p)\cong\frac{Q^{-1}(k\bb{t}[x_1,\ldots,x_n])}{(f_1+tg_1,\ldots,f_n+tg_n)}$ for some multiplicatively closed subset $Q\subset k\bb{t}[x_1,\ldots,x_n]$ (since localization commutes with quotients).

To prove the claim, we first note that $S$ is multiplicatively closed. Indeed, since $X^p$ is finite over a Noetherian base, $X^p$ is Noetherian. Thus $X^p$ has finitely many components, so $\min(R)^p$ is a finite set of primes. Moreover, $\mf{p}\in\Spec{S^{-1}R}$ if and only if $\mf{p}\in\min(R)^p$, so $\Spec{S^{-1}R}=X^p$.

The assumption that $p\in X_0$ is isolated implies that the local ring $\mc{O}_{X_0,p}$ has dimension 0. Note that $\mc{O}_{X_0,p}\cong\frac{Q^{-1}(k\bb{t}[x_1,\ldots,x_n])}{(f_1+tg_1,\ldots,f_n+tg_n,t)}$. Since $M:=Q^{-1}(k\bb{t}[x_1,\ldots,x_n])_\mf{m}$ is a regular local ring (of dimension $n+1$), it is a local Cohen--Macaulay ring. Thus \cite[\href{https://stacks.math.columbia.edu/tag/02NJ}{Lemma 02NJ}]{stacks} implies that $(f_1+tg_1,\ldots,f_n+tg_n,t)$ is a regular sequence in $M$. It follows that $(f_1+tg_1,\ldots,f_n+tg_n)$ is also a regular sequence in $M$. Since $S^{-1}R$ is already local by Proposition~\ref{prop:X^p local}, we also have $\mc{O}_{X^p}(X^p)\cong S^{-1}R\cong M/(f_1+tg_1,\ldots,f_n+tg_n)$.
\end{proof}

\begin{prop}\label{prop:X^p flat}
If $p\in X_0$ is isolated, then $X^p\to\Spec{k\bb{t}}$ is flat.
\end{prop}
\begin{proof}
Using the notation in the proof of Proposition~\ref{prop:X^p reg}, we have that $k\bb{t}$ is a regular local ring of dimension 1, $S^{-1}R$ is Cohen--Macaulay of dimension 1, and $S^{-1}R\otimes k\cong\mc{O}_{X_0,p}$ has dimension 0. Thus~\cite[Theorem 23.1 (p. 179)]{Mat89} implies that $X^p\to\Spec{k\bb{t}}$ is flat.
\end{proof}

\begin{prop}
If $p\in X_0$ is isolated, then $\mc{O}_{X^p}(X^p)$ is a free $k\bb{t}$-module.
\end{prop}
\begin{proof}
Since $k\bb{t}$ is Noetherian and $X^p\to \Spec{k\bb{t}}$ is finite, flatness (Proposition~\ref{prop:X^p flat}) implies that $\mc{O}_{X^p}(X^p)$ is a projective $k\bb{t}$-module. Since projective modules are locally free, $\mc{O}_{X^p}(X^p)$ being local (Proposition~\ref{prop:X^p local}) that $\mc{O}_{X^p}(X^p)$ is a free $k\bb{t}$-module.
\end{proof}

We can now define $\deg^{\mb{A}^1}_p(f_1+tg_1,\ldots,f_n+tg_n)\in\GW(k\bb{t})$.

\begin{defn}
Let $(f_1,\ldots,f_n):\mb{A}^n_k\to\mb{A}^n_k$ with isolated zero $p$. Let $g_1,\ldots,g_n$ be any elements of $k\bb{t}[x_1,\ldots,x_n]$ such that
\[\Spec\frac{k\bb{t}[x_1,\ldots,x_n]}{(f_1+tg_1,\ldots,f_n+tg_n)}\to\Spec{k\bb{t}}\]
is finite and flat. Let $X=\mb{V}(f_1+t_1g_1,\ldots,f_n+tg_n)$. Define $\deg_p^{\mb{A}^1}(f_1+tg_1,\ldots,f_n+tg_n)\in\GW(k\bb{t})$ to be the isomorphism class of the Scheja--Storch bilinear pairing $\mc{O}_{X^p}(X^p)\times\mc{O}_{X^p}(X^p)\to k\bb{t}$ determined by the regular sequence $(f_1+tg_1,\ldots,f_n+tg_n)$.
\end{defn}

Putting this all together, we get the following rephrasing of \cite[Theorem 5]{PW20}:

\begin{thm}[Dynamic local $\mb{A}^1$-degree]
\label{thm:dynamic degree}
Let $(f_1,\ldots,f_n):\mb{A}^n_k\to\mb{A}^n_k$ with isolated zero $p$. Let $g_1,\ldots,g_n$ be any elements of $k\bb{t}[x_1,\ldots,x_n]$ such that
\[\Spec\frac{k\bb{t}[x_1,\ldots,x_n]}{(f_1+tg_1,\ldots,f_n+tg_n)}\to\Spec{k\bb{t}}\]
is finite and flat. Let $X=\mb{V}(f_1+t_1g_1,\ldots,f_n+tg_n)$, and let $X^p_t:=(X^p)_t\subset\mb{A}^n_{k\pp{t}}$ be the collection of points that $p$ splits into under the deformation $X_0\mapsto X_t$. Then
\[\deg_p^{\mb{A}^1}(f_1,\ldots,f_n)=\jmath^{-1}|_{\mr{im}(\imath)}\left(\sum_{z\in X^p_t}\deg_z^{\mb{A}^1}(f_1+tg_1,\ldots,f_n+tg_n)\right)\]
as elements of $\GW(k)$.
\end{thm}
\begin{proof}
Let $f=(f_1,\ldots,f_n)$ and $g=(g_1,\ldots,g_n)$. By construction, we have $\mr{ev}_0(\deg_p^{\mb{A}^1}(f+tg))=\deg_p^{\mb{A}^1}(f)$ as elements of $\GW(k)$. The map $\mb{A}^n_{k\bb{t}}\to\mb{A}^n_{k\pp{t}}$ induced by the inclusion $k\bb{t}\hookrightarrow k\pp{t}$ sends $p\in X\subset\mb{A}^n_{k\bb{t}}$ to $X^p_t\subset\mb{A}^n_{k\pp{t}}$, so
\[\imath(\deg_p^{\mb{A}^1}(f+tg))=\sum_{z\in X^p_t}\deg_z^{\mb{A}^1}(f+tg)\]
as elements of $\GW(k\pp{t})$. The result now follows from Proposition~\ref{prop:triangle of gw}.
\end{proof}

Since transversality is a generic condition, Theorem~\ref{thm:dynamic degree} implies that we can always interpret the local index $\ind_q\sigma$ as a sum of local indices in the transverse setting, even when $q$ is not a simple zero of $\sigma$. By Proposition~\ref{prop:transverse}, we can always geometrically interpret the local index as a sum of intersection volumes. For example, Theorem~\ref{thm:dynamic degree} allows us to remove the transversality assumption in Theorem~\ref{thm:bezout}:

\begin{cor}\label{cor:bezout non-transverse}
Let $f=(f_1,\ldots,f_n):\mb{A}^n_k\to\mb{A}^n_k$ be a morphism with isolated zero $p$. Assume that $k(p)/k$ is separable. Let $g_1,\ldots,g_n\in k\bb{t}[x_1,\ldots,x_n]$ be such that the hypersurfaces $\mb{V}(f_i+tg_i)\subseteq\mb{P}^n_{k\pp{t}}$ meet transversely. Let $Y=\mb{V}(f_1+tg_1,\ldots,f_n+tg_n)\to\Spec{k\bb{t}}$. Then
\[\deg_p^{\mb{A}^1}(f_1,\ldots,f_n)=\sum_{z\in Y^p-\{p\}}\Tr_{\kappa(z)/k\pp{t}}\langle\Vol(z)\rangle,\]
\end{cor}
where $\Vol(z)$ is the intersection volume of $f_1+tg_1,\ldots,f_n+tg_n$ at $z$.
\begin{proof}
We first show that $\kappa(z)/k\pp{t}$ is separable for all $z\in Y^p-\{p\}$. Let $\Phi:Y^p\to\Spec{k\bb{t}}$ be the structure map, which is finite by Proposition~\ref{prop:X^p finite}. By \cite[\href{https://stacks.math.columbia.edu/tag/02GL}{Lemma 02GL (1)}]{stacks}, our assumption that $k(p)/k$ is separable implies that $p=\Spec{k(p)}$ is smooth over $k$. In particular, $\Phi$ is smooth at $\Phi^{-1}(0)=(Y^p)_0=p$. By \cite[\href{https://stacks.math.columbia.edu/tag/01V9}{Lemma 01V9}]{stacks}, there exists a non-empty open subset $U\subseteq Y^p$ such that $p\in U$ and $\Phi|_U$ is smooth. But $\Phi$ is proper and $Y^p-U$ is closed, so $\Phi(Y^p-U)\subseteq\Spec{k\bb{t}}$ is also closed. Any non-empty closed subset of $\Spec{k\bb{t}}$ contains the sole closed point 0. Since $p\not\in Y^p-U$, we have that $\Phi(Y^p-U)$ is empty and hence $Y^p=U$ (as $\Phi$ is surjective). It follows that $\Phi$ is smooth above the generic point, so $(Y^p)_t\to\Spec{k\pp{t}}$ is smooth. This map also inherits finiteness from $\Phi$, so $(Y^p)_t\to\Spec{k\pp{t}}$ is smooth of relative dimension 0 and is therefore \'etale. It now follows from \cite[\href{https://stacks.math.columbia.edu/tag/02GL}{Lemma 02GL (2)}]{stacks} that $\kappa(z)/k\pp{t}$ is separable for all $z\in(Y^p)_t=Y^p-\{p\}$.

Since we have assumed that $\mb{V}(f_i+tg_i)$ meet transversely, \cite[Section 3]{McK21} implies that $\deg_z^{\mb{A}^1}(f_1+tg_1,\ldots,f_n+tg_n)=\Tr_{\kappa(z)/k\pp{t}}\langle\Vol(z)\rangle$. The result now follows from Theorem~\ref{thm:dynamic degree}.
\end{proof}

\subsection{Computing the local degree in families}\label{sec:familial}
In essence, the dynamic approach enables us to compute the local $\mb{A}^1$-degree at a point by computing a sum of local $\mb{A}^1$-degrees over a nearby fiber. Using Harder's theorem~\cite[Lemma 30]{KW19}, we might instead try to compute $\deg_p^{\mb{A}^1}(f)$ by computing a sum of local $\mb{A}^1$-degrees over an arbitrary fiber in a family containing $p$. Since $\mb{V}(f)$ is zero dimensional, a family $X\to\Spec{k[t]}$ with special fiber $X_0=\mb{V}(f)$ is a branched cover of the affine line. We want to separate $p$, a point of higher intersection multiplicity, into a set of reduced points. We then wish to express $\deg_p^{\mb{A}^1}(f)$ as a sum of local $\mb{A}^1$-degrees over this set of reduced points.

However, if $X$ is ramified somewhere between the special fiber $X_0$ and the fiber over which we wish to compute the local $\mb{A}^1$-degree, then we may lose track of the individual points at which to compute --- there can be multiple points in the fiber $X_0$ that belong to the same connected component of $X$ (see Figure~\ref{fig:bad-harder}). We can avoid this issue by assuming that $X$ is ramified only at the fiber containing $p$ and removing the unwanted components of $X$ by localizing to the irreducible components of $X$ that contain $p$ (see Figure~\ref{fig:harder}). In a sense, this ramification assumption allows us to mimic the dynamic approach over the non-local base $\Spec{k[t]}$.

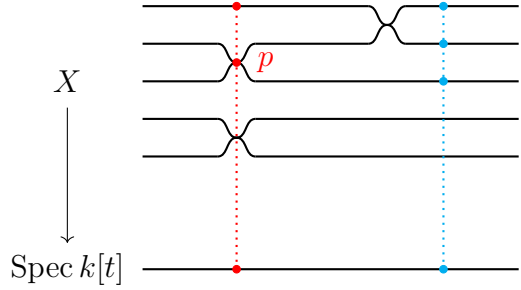
\begin{figure}
\begin{tikzpicture}
\draw (0,0) node {$\Spec{k[t]}$};
\draw (0,2.5) node {$X$};
\draw[<-] (0,0.35) -- (0,2.15);
\draw[thick] (1,3.5) -- (4,3.5) coordinate(e);
\draw[thick] (4.5,3.5) -- (6,3.5);
\draw[thick] (1,3) -- (2,3) coordinate(a);
\draw[thick] (2.5,3) -- (4,3) coordinate(f);
\draw[thick] (4.5,3) -- (6,3);
\draw[thick] (1,2.5) -- (2,2.5) coordinate(b);
\draw[thick] (2.5,2.5) -- (6,2.5);
\draw[thick] (1,2) -- (2,2) coordinate(c);
\draw[thick] (2.5,2) -- (6,2);
\draw[thick] (1,1.5) -- (2,1.5) coordinate(d);
\draw[thick] (2.5,1.5) -- (6,1.5);
\draw[thick] (1,0) -- (6,0);
\draw[thick]
  (a) ++(.25,-.25) coordinate(ab) to[out=180,in=0] (a)
  (ab) to[out=180,in=0] (b)
  (ab) to[out=0,in=180] ++(.25,.25)
  (ab) to[out=0,in=180] ++(.25,-.25)
  ;
\draw[thick]
  (c) ++(.25,-.25) coordinate(cd) to[out=180,in=0] (c)
  (cd) to[out=180,in=0] (d)
  (cd) to[out=0,in=180] ++(.25,.25)
  (cd) to[out=0,in=180] ++(.25,-.25)
  ;
\draw[thick]
  (e) ++(.25,-.25) coordinate(ef) to[out=180,in=0] (f)
  (ef) to[out=180,in=0] (e)
  (ef) to[out=0,in=180] ++(.25,.25)
  (ef) to[out=0,in=180] ++(.25,-.25)
  ;
\filldraw[red]
  (ab) circle(.05) node[anchor=west] {\ $p$}
  (2.25,3.5) circle(.05)
  (2.25,0) circle(.05)
  ;
\filldraw[cyan]
  (5,3.5) circle(.05)
  (5,3) circle(.05)
  (5,2.5) circle(.05)
  (5,0) circle(.05)
  ;
\draw[thick,cyan,dotted] (5,0) -- (5,3.5);
\draw[thick,red,dotted] (2.25,0) -- (2.25,3.5);
\end{tikzpicture}
\caption{Losing track of points that split off from $p$}\label{fig:bad-harder}
\end{figure}

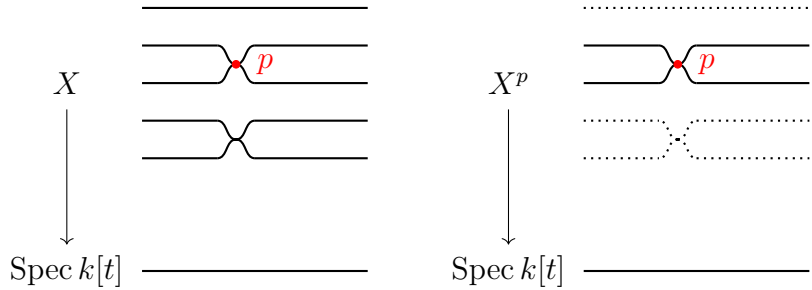
\begin{figure}
\begin{tikzpicture}
\draw (0,0) node {$\Spec{k[t]}$};
\draw (0,2.5) node {$X$};
\draw[<-] (0,0.35) -- (0,2.15);
\draw[thick] (1,3.5) -- (4,3.5);
\draw[thick] (1,3) -- (2,3) coordinate(a);
\draw[thick] (2.5,3) -- (4,3);
\draw[thick] (1,2.5) -- (2,2.5) coordinate(b);
\draw[thick] (2.5,2.5) -- (4,2.5);
\draw[thick] (1,2) -- (2,2) coordinate(c);
\draw[thick] (2.5,2) -- (4,2);
\draw[thick] (1,1.5) -- (2,1.5) coordinate(d);
\draw[thick] (2.5,1.5) -- (4,1.5);
\draw[thick] (1,0) -- (4,0);
\draw[thick]
  (a) ++(.25,-.25) coordinate(ab) to[out=180,in=0] (a)
  (ab) to[out=180,in=0] (b)
  (ab) to[out=0,in=180] ++(.25,.25)
  (ab) to[out=0,in=180] ++(.25,-.25)
  ;
\draw[thick]
  (c) ++(.25,-.25) coordinate(cd) to[out=180,in=0] (c)
  (cd) to[out=180,in=0] (d)
  (cd) to[out=0,in=180] ++(.25,.25)
  (cd) to[out=0,in=180] ++(.25,-.25)
  ;
\filldraw[red]
  (ab) circle(.05) node[anchor=west] {\ $p$}
  ;
\end{tikzpicture}
\qquad
\begin{tikzpicture}
\draw (0,0) node {$\Spec{k[t]}$};
\draw (0,2.5) node {$X^p$};
\draw[<-] (0,0.35) -- (0,2.15);
\draw[thick,dotted] (1,3.5) -- (4,3.5);
\draw[thick] (1,3) -- (2,3) coordinate(a);
\draw[thick] (2.5,3) -- (4,3);
\draw[thick] (1,2.5) -- (2,2.5) coordinate(b);
\draw[thick] (2.5,2.5) -- (4,2.5);
\draw[thick,dotted] (1,2) -- (2,2) coordinate(c);
\draw[thick,dotted] (2.5,2) -- (4,2);
\draw[thick,dotted] (1,1.5) -- (2,1.5) coordinate(d);
\draw[thick,dotted] (2.5,1.5) -- (4,1.5);
\draw[thick] (1,0) -- (4,0);
\draw[thick]
  (a) ++(.25,-.25) coordinate(ab) to[out=180,in=0] (a)
  (ab) to[out=180,in=0] (b)
  (ab) to[out=0,in=180] ++(.25,.25)
  (ab) to[out=0,in=180] ++(.25,-.25)
  ;
\draw[thick,dotted]
  (c) ++(.25,-.25) coordinate(cd) to[out=180,in=0] (c)
  (cd) to[out=180,in=0] (d)
  (cd) to[out=0,in=180] ++(.25,.25)
  (cd) to[out=0,in=180] ++(.25,-.25)
  ;
\filldraw[red]
  (ab) circle(.05) node[anchor=west] {\ $p$}
  ;
\end{tikzpicture}
\caption{Removing disjoint sheets}\label{fig:harder}
\end{figure}

For now, we will not assume that $X$ is branched only at one point. Instead, we will show how to compute a sum of local degrees via Harder's theorem after localizing to a connected component. One can then apply this more general result to the special case where $X$ is branched at only one point. Before describing the familial local $\mb{A}^1$-degree (Theorem~\ref{thm:harder}), we need the following analog of Proposition~\ref{prop:X^p finite}:

\begin{lem}\label{lem:finiteness}
Let $\vphi:X\to\Spec{k[t]}$ be a morphism of finite type, where $X$ is affine. Assume that every irreducible component of $X$ surjects onto $\Spec{k[t]}$ under $\vphi$, that $\vphi$ is unramified away from the preimage of a finite locus $B\subset\Spec{k[t]}$, and $\vphi^{-1}(B)$ is finite. Then $\vphi$ is finite and flat.
\end{lem}
\begin{proof}
We will first show that $\vphi$ is flat. Write $X=\Spec{A}$ for some $k[t]$-module $A$. Since $k[t]$ is a Dedekind domain, it suffices to show that $A$ is torsion-free. Suppose $g\in k[t]$ is a non-zero element that annihilates some $a\in A$. Then for any irreducible component $Y\subseteq X$ on which $a$ does not vanish, we have $\vphi(Y)\subseteq\mb{V}(g)\subsetneq\Spec{k[t]}$. But this contradicts our assumption that each irreducible component of $X$ surjects onto $\Spec{k[t]}$, so we deduce that $\vphi$ is flat.

Next, we show that $\vphi$ has finite fibers. Since $\vphi$ is affine and finite type, $\vphi$ is quasi-finite if and only if it has finite fibers~\cite[\href{https://stacks.math.columbia.edu/tag/02NH}{Lemma 02NH}]{stacks}; the same is also true for the restriction of $\vphi$ to $\vphi':X-\vphi^{-1}(B)\to\mb{A}^1_k-B$. The map $\vphi'$ is unramified by assumption and is therefore locally quasi-finite by~\cite[\href{https://stacks.math.columbia.edu/tag/02V5}{Lemma 02V5}]{stacks}. Since $\vphi'$ is affine and hence quasi-compact \cite[\href{https://stacks.math.columbia.edu/tag/01S7}{Lemma 01S7}]{stacks}, it follows that $\vphi'$ is quasi-finite~\cite[\href{https://stacks.math.columbia.edu/tag/01TJ}{Lemma 01TJ}]{stacks}. Thus $\vphi'$ has finite fibers. The fibers of $\vphi$ above $B$ are finite by assumption, so $\vphi$ has finite fibers.

Finally, note that if $Z$ is an irreducible component of $X-\vphi^{-1}(B)$, then each fiber of $Z$ consists of a single point. Indeed, $Z$ is connected (being irreducible), so if some fiber of $Z$ consists of more than one point, then $Z$ consists of more than one sheet. But $Z\to\mb{A}^1_k-B$ is unramified, so these sheets must remain disjoint. This contradicts the assumption that $Z$ is irreducible. Thus $Z\to\mb{A}^1_k-B$ is injective, so this map is an isomorphism. It follows that the closure of $Z$ in $X$ is isomorphic to $\mb{A}^1_k$, so $X$ is a finite union of isomorphic copies of $\Spec{k[t]}$. As a result, $X$ is finite over $\Spec{k[t]}$.
\end{proof}

\begin{thm}[Familial local $\mb{A}^1$-degree]\label{thm:harder}
Let $f:\mb{A}^n_k\to\mb{A}^n_k$ such that each point of $f^{-1}(0)$ is isolated in the fiber. Let $F:\mb{A}^n_{k[t]}\to\mb{A}^n_{k[t]}$ be a morphism such that $\mb{V}(F)\to\Spec{k[t]}$ is flat and $F|_{t=0}=f$. Assume that $\mb{V}(F)$ is unramified away from a finite set $B\subset\mb{A}^1_k$. Then for any closed point $c\in\mb{A}^1_k-\{0\}$ and any connected component $Y\subseteq\mb{V}(F)$, the perturbation $\tilde{f}:=F|_{t=c}:\mb{A}^n_k\to\mb{A}^n_k$ of $f$ has a set of zeros $Z\subseteq\tilde{f}^{-1}(0)$ such that
\[\sum_{p\in Y_0}\deg^{\mb{A}^1}_p(f)=\sum_{q\in Z}\deg_q^{\mb{A}^1}(\tilde{f}).\]
\end{thm}
\begin{proof}
We will construct a pair $(Q,\beta)$, where $Q$ is a finite locally free $k[t]$-module and $\beta$ is a non-degenerate symmetric bilinear form on $Q$, such that
\begin{enumerate}[(i)]
\item the isomorphism class of $\beta|_{t=0}$ is $\sum_{p\in Y_0}\deg^{\mb{A}^1}_p(f)$, and
\item the isomorphism class of $\beta|_{t=c}$ is $\sum_{q\in Z}\deg^{\mb{A}^1}_q(\tilde{f})$.
\end{enumerate}
Once we have constructed $(Q,\beta)$, the desired result will follow from~\cite[Lemma 30]{KW19}.

Since $\mb{V}(F)\to\Spec{k[t]}$ is finite type and $\Spec{k[t]}$ is Noetherian, we have that $\mb{V}(F)$ is Noetherian as well~\cite[\href{https://stacks.math.columbia.edu/tag/01T6}{Lemma 01T6}]{stacks}. In particular, $\mb{V}(F)$ has finitely many irreducible components. Let $P\subset k[t][x_1,\ldots,x_n]$ denote the (finite) set of prime ideals corresponding to the irreducible components comprising $Y$. Then $S=k[t][x_1,\ldots,x_n]-P$ is multiplicatively closed. Set $Q=\frac{S^{-1}(k[t][x_1,\ldots,x_n])}{(F_1,\ldots,F_n)}$. The localization $\Spec{Q}$ is the restriction of the vanishing locus $\mb{V}(F)$ to the connected component $Y$.

The conditions of Lemma~\ref{lem:finiteness} hold for $Y$, which implies that $Q$ is a finite $k[t]$-module. Since $\Spec{Q}\to\mb{V}(F)$ is flat by \cite[\href{https://stacks.math.columbia.edu/tag/00HT}{Lemma 00HT (1)}]{stacks} and $\mb{V}(F)\to\Spec{k[t]}$ is flat by assumption, \cite[\href{https://stacks.math.columbia.edu/tag/01U7}{Lemma 01U7}]{stacks} implies that $Q$ is a flat $k[t]$-module. Since $k[t]$ is Noetherian, $Q$ being a finite $k[t]$-module is equivalent to $Q$ being a finitely presented $k[t]$-module, so \cite[\href{https://stacks.math.columbia.edu/tag/00NX}{Lemma 00NX (1) and (7)}]{stacks} implies that $Q$ is a finite locally free $k[t]$-module. (In fact, $Q$ is projective over a PID, so $Q$ is even a free $k[t]$-module.) 

We thus have the desired $Q$. We define $\beta$ to be the Scheja--Storch form on $Q$ associated to the sequence $(F_1,\ldots,F_n)$. Then $Q_0$ and $Q_c$ each have finite $k$-dimension. This implies that $Q_0$ and $Q_c$ are Artinian rings, each having finitely many maximal ideals. The maximal ideals of $Q_0$ correspond to the points of $Y_0$ (which are the zeros of $f$ contained in $Y$), while the maximal ideals of $Q_c$ are a subset of the zeros of $\tilde{f}$. Let $Z\subset\mb{A}^n_k$ be the set of points corresponding to the maximal ideals of $Q_c$. It follows from e.g.~\cite[Lemma 4.7 and Theorem 5.1]{BMP21} that $\beta|_{t=0}$ is isomorphic to $\sum_{p\in Y_0}\deg_p^{\mb{A}^1}(f)$ and $\beta|_{t=c}$ is isomorphic to $\sum_{q\in Z}\deg_q^{\mb{A}^1}(\tilde{f})$, which gives us (i) and (ii).
\end{proof}

\begin{cor}[Familial local $\mb{A}^1$-degree at one point]\label{cor:harder}
Assume the conventions of Theorem~\ref{thm:harder}. Assume moreover that $B=\{0\}$. Fix $p\in f^{-1}(0)$. Then for any closed point $c\in\mb{A}^1_k$, the perturbation $\tilde{f}:=F_{t=c}$ of $f$ has a set of zeros $Z\subseteq\tilde{f}^{-1}(0)$ such that
\[\deg_p^{\mb{A}^1}(f)=\sum_{q\in Z}\deg_q^{\mb{A}^1}(\tilde{f}).\]
\end{cor}
\begin{proof}
By assumption, $\mb{V}(F)-\vphi^{-1}(B)$ is a disjoint union of copies of the punctured affine line. The closure in $\mb{V}(F)$ of each irreducible component of $\mb{V}(F)-\vphi^{-1}(B)$ has a single point in the fiber above $t=0$, so the connected components are in bijection with the points in the fiber $\vphi^{-1}(B)$. Let $Y$ be the connected component of $\mb{V}(F)$ containing $p$. The previous discussion implies that $Y_0=\{p\}$, so the result follows from Theorem~\ref{thm:harder}.
\end{proof}

\begin{rem}
Similar to Theorem~\ref{thm:harder}, Kass and Wickelgren have used Harder's theorem to study the $\mb{A}^1$-degree in families~\cite{KW19,KW21}. In their work, they show (and utilize) that the sum of local $\mb{A}^1$-degrees over a given fiber is independent of the fiber chosen. Our approach describes how to remove other elements of the fiber over 0 in order to compute the local $\mb{A}^1$-degree at a subset of the fiber by working in families. Under the assumptions of Corollary~\ref{cor:harder}, this can be refined to compute the local $\mb{A}^1$-degree at a single point.
\end{rem}

As with the dynamic approach, Corollary~\ref{cor:harder} allows us to remove the transversality assumption in Theorem~\ref{thm:bezout}:

\begin{cor}
Assume the conventions of Corollary~\ref{cor:harder}. Assume moreover that away from $t=0$, each fiber $\mb{V}(F)_t$ is geometrically reduced. Then for any $c\in k^\times$, the perturbation $\tilde{f}:=F|_{t=c}:\mb{A}^n_k\to\mb{A}^n_k$ of $f$ has a set of zeros $Z\subseteq\tilde{f}^{-1}(0)$ such that
\[\deg^{\mb{A}^1}_p(f)=\sum_{q\in Z}\Tr_{k(q)/k}\langle\Vol(q)\rangle,\]
where $\Vol(q)=\Jac(\tilde{f})(q)$.
\end{cor}
\begin{proof}
By assumption, $\Spec{Q}_c$ is geometrically reduced, so the components of $\tilde{f}$ meet transversely at each $q\in Z$. Since $F$ is flat and unramified at $t=c$, we have that $\mb{V}(\tilde{f})\to\Spec{k(c)}=\Spec{k}$ is \'etale~\cite[\href{https://stacks.math.columbia.edu/tag/02GU}{Lemma 02GU (2) and (4)}]{stacks}. In particular, $k(q)/k$ is separable for all $q\in Z$~\cite[\href{https://stacks.math.columbia.edu/tag/02GL}{Lemma 02GL (1)}]{stacks}. It follows from~\cite[Section 5.2]{McK21} that $\deg_q^{\mb{A}^1}(\tilde{f})=\Tr_{k(q)/k}\langle\Vol(q)\rangle$. The desired result now follows directly from Corollary~\ref{cor:harder}.
\end{proof}

In summary, the intersection volume is a universal geometric interpretation of the local indices in $\mb{A}^1$-enumerative geometry. However, for most enumerative geometric problems, this interpretation is unsatisfactory --- for the circles of Apollonius, the intersection volume at a tangent circle $C(q)$ would tell us about the geometry of the cones $Q(p_i)$ (or planes $V(p_i)$), rather than about the geometry of the circles $C(p_i)$ (or points $p_i$) and the tangent circle $C(q)$. Question~\ref{ques:geometric interpretation} asks for a more intrinsic geometric interpretation of $\ind_q\sigma$.

\section{Local contributions for Apollonian circles}\label{sec:local contributions apollonius}
We now give a geometric interpretation of $\ind_q\sigma$ in terms of the geometry of the relevant circles by analyzing the intersection volume. We will assume that the cones $Q(p_i)$ intersect transversely, which happens whenever the circles $C(p_i)$ satisfy the criteria of Proposition~\ref{prop:transverse cones}. (For example, the centers of all three circles should not lie on a shared line.) In Section~\ref{sec:other invariants}, we will outline a conjectural approach to finding alternative, more parsimonious descriptions of $\ind_q\sigma$.

Given three circles $C(p_i)$ (or points $p_i$), the intersection volume $\Vol(q)$ at a circle $C(q)$ is defined in terms of the gradients of the cones $Q(p_i)$ (or planes $V(p_i)$) at $q$. We will assume that $C(p_i)$ and $C(q)$ are non-degenerate circles, so that their $c_0$ coordinate in $\mb{P}^3_k$ is non-zero. This allows us to work in the affine patch $\{c_0\neq 0\}\subset\mb{P}^3_k$, where the twisted covering map of \cite[Proposition 3.8]{McK21} is simply the standard covering map $\{c_0\neq 0\}\to\mb{A}^3_k$. The standard coordinates on $\{c_0\neq 0\}$ are $(\frac{c_1}{c_0},\frac{c_2}{c_0},\frac{c_3}{c_0})$, so the gradient used to calculate $\Vol(q)$ will be $\nabla=(\frac{\partial}{\partial c_1},\frac{\partial}{\partial c_2},\frac{\partial}{\partial c_3})$.

\begin{notn}\label{notn:u and v}
Let $z_i:=\pcoor{a_i:b_i:1}$ be the center of $C(p_i)$ (or the point $p_i$), and let $r_i^2$ be the radius squared of $C(p_i)$ (or 0 for the point $p_i$). Similarly, let $\gamma:=\pcoor{\alpha:\beta:1}$ be the center of the non-degenerate circle $C(q)$. Let $\rho^2$ be the radius squared of $C(q)$, which is 0 if $C(q)$ is simply the point $\pcoor{\alpha:\beta:1}$. Let $\tau_i:=\pcoor{s_i:t_i:1}\in C(p_i)\cap C(q)$ be the point at which $C(p_i)$ and $C(q)$ are tangent.

We will use the following vectors in $\mb{A}^2_k$ (see Figure~\ref{fig:vectors}):
\begin{align*}
\vv{\gamma z_i}&=(a_i-\alpha,b_i-\beta),\\
\vv{\gamma\tau_i}&=(s_i-\alpha,t_i-\beta),\\
\vv{\tau_i z_i}&=(a_i-s_i,b_i-t_i).
\end{align*}

\begin{figure}
    \centering
    \begin{subfigure}{.3\linewidth}
    \centering
    \begin{tikzpicture}
    \draw[color=red, very thick](0,0) circle (1);
    \draw[color=cyan, very thick] ({sqrt(.5)+sqrt((.5^2)/2)},{sqrt(.5)+sqrt((.5^2)/2)}) circle (.5);
    \draw[->,thick] (0,0) -- ({sqrt(.5)+sqrt((.5^2)/2)},{sqrt(.5)+sqrt((.5^2)/2)});
    \fill (0,0) circle[radius=1pt];
    \node at (0,-1/3) {$\vv{\gamma z_i}$};
    \node[text=red] at (-1.5,-1/2) {$C(q)$};
    \node[text=cyan] at ({sqrt(.5)+sqrt((.5^2)/2)+1},{sqrt(.5)+sqrt((.5^2)/2)-1/2}) {$C(p_i)$};
    \end{tikzpicture}
    \end{subfigure}
    \begin{subfigure}{.3\linewidth}
    \centering
    \begin{tikzpicture}
    \draw[color=red, very thick](0,0) circle (1);
    \draw[color=cyan, very thick] ({sqrt(.5)+sqrt((.5^2)/2)},{sqrt(.5)+sqrt((.5^2)/2)}) circle (.5);
    \draw[->,thick] (0,0) -- ({sqrt(.5)},{sqrt(.5)});
    \fill (0,0) circle[radius=1pt];
    \node at (0,-1/3) {$\vv{\gamma\tau_i}$};
    \node[text=red] at (-1.5,-1/2) {$C(q)$};
    \node[text=cyan] at ({sqrt(.5)+sqrt((.5^2)/2)+1},{sqrt(.5)+sqrt((.5^2)/2)-1/2}) {$C(p_i)$};
    \end{tikzpicture}
    \end{subfigure}
    \begin{subfigure}{.3\linewidth}
    \centering
    \begin{tikzpicture}
    \draw[color=red, very thick](0,0) circle (1);
    \draw[color=cyan, very thick] ({sqrt(.5)+sqrt((.5^2)/2)},{sqrt(.5)+sqrt((.5^2)/2)}) circle (.5);
    \draw[->,thick] ({sqrt(.5)},{sqrt(.5)}) -- ({sqrt(.5)+sqrt((.5^2)/2)},{sqrt(.5)+sqrt((.5^2)/2)});
    \fill ({sqrt(.5)},{sqrt(.5)}) circle[radius=1pt];
    \node at ({sqrt(.5)-1/3},{sqrt(.5)-1/3}) {$\vv{\tau_i z_i}$};
    \node[text=red] at (-1.5,-1/2) {$C(q)$};
    \node[text=cyan] at ({sqrt(.5)+sqrt((.5^2)/2)+1},{sqrt(.5)+sqrt((.5^2)/2)-1/2}) {$C(p_i)$};
    \end{tikzpicture}
    \end{subfigure}
    \caption{Externally tangent circles}
    \label{fig:vectors}
\end{figure}
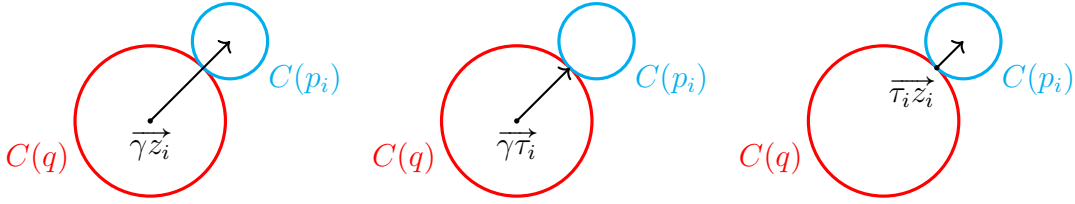

Finally, define
\begin{align*}
u_i&=\begin{cases}\vv{\tau_i z_i}\cdot\vv{\gamma z_i} & C(p_i)\text{ a circle},\\
1 & p_i\text{ a point}
\end{cases}\\
v_i&=\begin{cases}\vv{\tau_i z_i}\cdot\vv{\gamma \tau_i} & C(p_i)\text{ a circle},\\
1 & p_i\text{ a point}.
\end{cases}
\end{align*}
\end{notn}

\begin{rem}\label{rem:u_i and v_i}
If $k$ is an ordered field and $r_i^2,\rho^2>0$, then we can choose distinguished radii $r_i\in k(\sqrt{r_i^2})$ and $\rho\in k(\sqrt{\rho^2})$ such that $r_i,\rho>0$. Since the vectors $\vv{\gamma z_i}$, $\vv{\gamma\tau_i}$, and $\vv{\tau z_i}$ are all parallel or anti-parallel, the sign of the dot product of any two of these vectors indicates whether they are parallel or anti-parallel.

In this context, $v_i$ detects whether $C(p_i)$ and $C(q)$ are \textit{externally} tangent (as in Figure~\ref{fig:vectors}) or \textit{internally} tangent (as in Figures~\ref{fig:int tangent} and~\ref{fig:int tangent flip}). Moreover, if $C(p_i)$ and $C(q)$ are internally tangent, then $u_i$ detects whether $\rho>r_i$ (as in Figure~\ref{fig:int tangent}) or $r_i>\rho$ (as in Figure~\ref{fig:int tangent flip}). In particular:
\begin{itemize}
\item $C(p_i)$ and $C(q)$ are externally tangent if and only if $u_i,v_i>0$.
\item $C(p_i)$ and $C(q)$ are internally tangent with $\rho>r_i$ if and only if $u_i<0$ and $v_i<0$.
\item $C(p_i)$ and $C(q)$ are internally tangent with $r_i>\rho$ if and only if $u_i>0$ and $v_i<0$.
\end{itemize}
\end{rem}

\begin{figure}
    \centering
    \begin{subfigure}{.3\linewidth}
    \centering
    \begin{tikzpicture}
    \draw[color=red, very thick](0,0) circle (1);
    \draw[color=cyan, very thick] ({sqrt(.5)-sqrt((.75^2)/2)},{sqrt(.5)-sqrt((.75^2)/2)}) circle (.75);
    \draw[->,thick] (0,0) -- ({sqrt(.5)-sqrt((.75^2)/2)},{sqrt(.5)-sqrt((.75^2)/2)});
    \fill (0,0) circle[radius=1pt];
    \node at ({sqrt(.5)-sqrt((.75^2)/2)},-1/3) {$\vv{\gamma z_i}$};
    \node[text=red] at (-1.5,-1/2) {$C(q)$};
    \node[text=cyan] at (1.5,1/2) {$C(p_i)$};
    \end{tikzpicture}
    \end{subfigure}
    \begin{subfigure}{.3\linewidth}
    \centering
    \begin{tikzpicture}
    \draw[color=red, very thick](0,0) circle (1);
    \draw[color=cyan, very thick] ({sqrt(.5)-sqrt((.75^2)/2)},{sqrt(.5)-sqrt((.75^2)/2)}) circle (.75);
    \draw[->,thick] (0,0) -- ({sqrt(.5)},{sqrt(.5)});
    \fill (0,0) circle[radius=1pt];
    \node at ({sqrt(.5)-sqrt((.75^2)/2)},-1/3) {$\vv{\gamma\tau_i}$};
    \node[text=red] at (-1.5,-1/2) {$C(q)$};
    \node[text=cyan] at (1.5,1/2) {$C(p_i)$};
    \end{tikzpicture}
    \end{subfigure}
    \begin{subfigure}{.3\linewidth}
    \centering
    \begin{tikzpicture}
    \draw[color=red, very thick](0,0) circle (1);
    \draw[color=cyan, very thick] ({sqrt(.5)-sqrt((.75^2)/2)},{sqrt(.5)-sqrt((.75^2)/2)}) circle (.75);
    \draw[->,thick] ({sqrt(.5)},{sqrt(.5)}) -- ({sqrt(.5)-sqrt((.75^2)/2)},{sqrt(.5)-sqrt((.75^2)/2)});
    \fill ({sqrt(.5)},{sqrt(.5)}) circle[radius=1pt];
    \node at ({sqrt(.5)-sqrt((.75^2)/2)},-1/5) {$\vv{\tau_i z_i}$};
    \node[text=red] at (-1.5,-1/2) {$C(q)$};
    \node[text=cyan] at (1.5,1/2) {$C(p_i)$};
    \end{tikzpicture}
    \end{subfigure}
    \caption{Internally tangent circles}
    \label{fig:int tangent}
\end{figure}
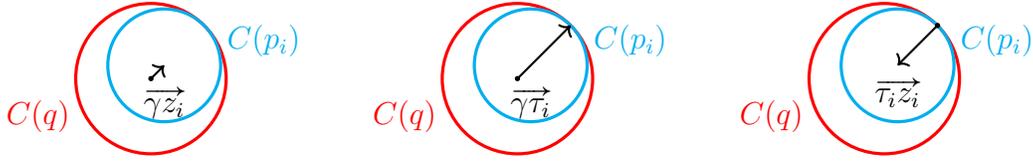

\begin{figure}
    \centering
    \begin{subfigure}{.3\linewidth}
    \centering
    \begin{tikzpicture}
    \draw[color=cyan, very thick](0,0) circle (1);
    \draw[color=red, very thick] ({sqrt(.5)-sqrt((.75^2)/2)},{sqrt(.5)-sqrt((.75^2)/2)}) circle (.75);
    \draw[->,thick] ({sqrt(.5)-sqrt((.75^2)/2)},{sqrt(.5)-sqrt((.75^2)/2)}) -- (0,0);
    \fill ({sqrt(.5)-sqrt((.75^2)/2)},{sqrt(.5)-sqrt((.75^2)/2)}) circle[radius=1pt];
    \node at ({sqrt(.5)-sqrt((.75^2)/2)},-1/3) {$\vv{\gamma z_i}$};
    \node[text=cyan] at (-1.5,-1/2) {$C(p_i)$};
    \node[text=red] at (1.5,1/2) {$C(q)$};
    \end{tikzpicture}
    \end{subfigure}
    \begin{subfigure}{.3\linewidth}
    \centering
    \begin{tikzpicture}
    \draw[color=cyan, very thick](0,0) circle (1);
    \draw[color=red, very thick] ({sqrt(.5)-sqrt((.75^2)/2)},{sqrt(.5)-sqrt((.75^2)/2)}) circle (.75);
    \draw[->,thick] ({sqrt(.5)-sqrt((.75^2)/2)},{sqrt(.5)-sqrt((.75^2)/2)}) -- ({sqrt(.5)},{sqrt(.5)});
    \fill ({sqrt(.5)-sqrt((.75^2)/2)},{sqrt(.5)-sqrt((.75^2)/2)}) circle[radius=1pt];
    \node at ({sqrt(.5)-sqrt((.75^2)/2)},-1/5) {$\vv{\gamma\tau_i}$};
    \node[text=cyan] at (-1.5,-1/2) {$C(p_i)$};
    \node[text=red] at (1.5,1/2) {$C(q)$};
    \end{tikzpicture}
    \end{subfigure}
    \begin{subfigure}{.3\linewidth}
    \centering
    \begin{tikzpicture}
    \draw[color=cyan, very thick](0,0) circle (1);
    \draw[color=red, very thick] ({sqrt(.5)-sqrt((.75^2)/2)},{sqrt(.5)-sqrt((.75^2)/2)}) circle (.75);
    \draw[->,thick] ({sqrt(.5)},{sqrt(.5)}) -- (0,0);
    \fill ({sqrt(.5)},{sqrt(.5)}) circle[radius=1pt];
    \node at ({sqrt(.5)-sqrt((.75^2)/2)},-1/4) {$\vv{\tau_i z_i}$};
    \node[text=cyan] at (-1.5,-1/2) {$C(p_i)$};
    \node[text=red] at (1.5,1/2) {$C(q)$};
    \end{tikzpicture}
    \end{subfigure}
    \caption{Internally tangent circles with reversed containment}
    \label{fig:int tangent flip}
\end{figure}
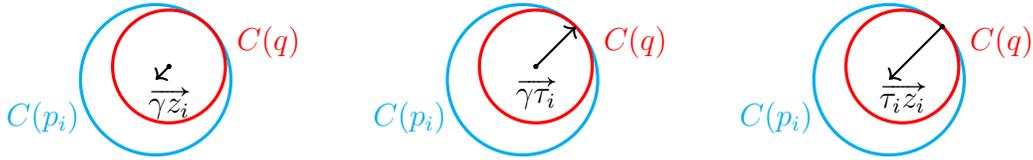

\begin{lem}\label{lem:geometric interpretation}
If $C(q)$ is tangent to the circles $C(p_i)$ (or points $p_i$), then the intersection volume is (up to squares)
\begin{align*}
\Vol(q)&=\sum_{\substack{\{i,m,n\}=\{1,2,3\}\\m<n}} (-1)^{i+1}u_iv_mv_n((a_m-\alpha)(b_n-\beta)-(a_n-\alpha)(b_m-\beta)).
\end{align*}
In other words, $\Vol(q)$ is a weighted sum of the signed areas of the parallelograms spanned by $\vv{\gamma z_m}$ and $\vv{\gamma z_n}$ (see Figure~\ref{fig:area}), where the weights are given in terms of the dot products $u_i$, $v_m$, and $v_n$.
\end{lem}
\begin{proof}
If $r_i^2\neq 0$, we have
\begin{align*}
\nabla Q(p_i)=(&2a_i(a_iX+b_iY+Z)-2r_i^2X,\\
&2b_i(a_iX+b_iY+Z)-2r_i^2Y,\\
&2(a_iX+b_iY+Z)).
\end{align*}
Evaluated at $q$, we have $X=2(a_i-\alpha)$, $Y=2(b_i-\beta)$, and $Z=\alpha^2-a_i^2+\beta^2-b_i^2+r_i^2-\rho^2$. Thus, evaluated at $q$, we have
\begin{align*}
\nabla Q(p_i)|_q=(&2a_i((a_i-\alpha)^2+(b_i-\beta)^2+r_i^2-\rho^2)-4r_i^2(a_i-\alpha),\\
&2b_i((a_i-\alpha)^2+(b_i-\beta)^2+r_i^2-\rho^2)-4r_i^2(b_i-\beta),\\
&2((a_i-\alpha)^2+(b_i-\beta)^2+r_i^2-\rho^2)).
\end{align*}
If $p_i=\pcoor{a_i:b_i:1}$ is a point, then $\nabla V(p_i)=(a_i,b_i,1)$ is independent of the intersection point $q$. The intersection volume $\Vol(q)$ is the determinant of the matrix $M$ with rows $\nabla Q(p_i)|_q$ (or $\nabla V(p_i)$). Subtracting $\alpha$ times the third column from the first column of $M$ and $\beta$ times the third column from the second column of $M$, we find that $\Vol(q)$ is the determinant of the matrix with $i\textsuperscript{th}$ row
\begin{align}\label{eq:row}
(&2(a_i-\alpha)((a_i-\alpha)^2+(b_i-\beta)^2-r_i^2-\rho^2),\\
&2(b_i-\beta)((a_i-\alpha)^2+(b_i-\beta)^2-r_i^2-\rho^2),\nonumber\\
&2((a_i-\alpha)^2+(b_i-\beta)^2+r_i^2-\rho^2))\nonumber
\end{align}
if $C(p_i)$ is a circle or
\[(a_i-\alpha,b_i-\beta,1)=(v_i(a_i-\alpha),v_i(b_i-\beta),u_i)\]
if $p_i$ is a point. Since $\pcoor{s_i:t_i:1}\in C(p_i)\cap C(q)$, we have $r_i^2=(s_i-a_i)^2+(t_i-b_i)^2$ and $\rho^2=(s_i-\alpha)^2+(t_i-\beta)^2$. If $C(p_i)$ is a circle, we may thus substitute for $r_i^2$ and $\rho^2$ in Equation~\ref{eq:row} to obtain $4(v_i(a_i-\alpha),v_i(b_i-\beta),u_i)$. Ignoring the factor of 4 only changes $\Vol(q)$ up to squares, so
\[\Vol(q)=\det\begin{pmatrix}
v_1(a_1-\alpha) & v_1(b_1-\beta) & u_1\\
v_2(a_2-\alpha) & v_2(b_2-\beta) & u_2\\
v_3(a_3-\alpha) & v_3(b_3-\beta) & u_3
\end{pmatrix}\]
up to squares.
\end{proof}

\begin{figure}
    \centering
    \begin{tikzpicture}
    \draw[very thick](0,0) circle (1);
    \draw[color=cyan, very thick] (1.75,0) circle (.75);
    \draw[color=red, very thick] ({-sqrt(.5)-sqrt((.5^2)/2)},{sqrt(.5)+sqrt((.5^2)/2)}) circle (.5);
    \draw[color=black, opacity=0.15, fill=black, fill opacity=0.15] (0,0) -- (1.75,0) -- ({1.75-sqrt(.5)-sqrt((.5^2/2)},{sqrt(.5)+sqrt((.5^2)/2)}) -- ({-sqrt(.5)-sqrt((.5^2)/2)},{sqrt(.5)+sqrt((.5^2)/2)}) -- cycle;
    \draw[->,color=cyan, thick] (0,0) -- (1.75,0);
    \draw[->,color=red, thick] (0,0) -- ({-sqrt(.5)-sqrt((.5^2)/2)},{sqrt(.5)+sqrt((.5^2)/2)});
    \node[text=cyan] at (1/2,-1/3) {$\vv{\gamma z_n}$};
    \node[text=red] at (-1/2,-1/3) {$\vv{\gamma z_m}$};
    \node[text=red] at (-2,1/2) {$C(p_m)$};
    \node[text=cyan] at (3,-1/2) {$C(p_n)$};
    \end{tikzpicture}
    \caption{Parallelogram of tangent circles}
    \label{fig:area}
\end{figure}
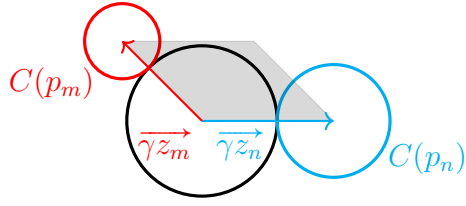

\section{Invariants from dual circles}\label{sec:other invariants}
Classically, the circles of Apollonius can be studied in pairs. One way to pair circles is via inversion~\cite[pp. 154--160]{Dor65}. Given three general circles $C_1,C_2,C_3$ (over $\mb{R}$), the \textit{radical circle} is the unique circle that intersects each $C_i$ perpendicularly. Inversion through the radical circle preserves tangency to each $C_i$ and hence determines a permutation on the set of circles tangent to $C_1,C_2,C_3$. One can then show that this permutation is in fact an involution. We say that two circles are \textit{inversively dual} or \textit{conjugate} if they correspond to one another under this involution.

Alternatively, one can degenerate $C_i$ to a point $p$. Under this degeneration, each tangent circle will be sent to one of the four circles through $p$ and tangent to the remaining two circles $C_j,C_k$. Moreover, the map (induced by degenerating $C_i$ to $p$) from the set of circles tangent to $C_1,C_2,C_3$ to the set of circles through $p$ and tangent to $C_j,C_k$ is two-to-one, so swapping the elements within each fiber determines an involution on the circles of Apollonius~\cite[pp. 117--121]{Joh60}. We say that two circles are \textit{degeneratively dual through $C_i$} if they coincide after degenerating $C_i$.

Over an ordered field, each of these involutions can be rephrased in terms of external and internal tangency (see Remark~\ref{rem:u_i and v_i}). Two tangent circles are degeneratively dual through $C_i$ if they have opposite tangencies with respect to $C_i$ and the same tangencies with respect to the remaining two circles $C_j,C_k$ (see Figure~\ref{fig:degen}). Two tangent circles are inversively dual if they have opposite tangencies with respect to $C_1,C_2,C_3$  (see Figure~\ref{fig:inversive}).

\begin{figure}
\begin{tikzpicture}
	\draw[very thick] (0,0) circle (.4);
    \draw[very thick] (1.05,.15) circle (.25);
    \draw[very thick] (.7,.8) circle (.1);
    \draw[color=red, very thick] (.117,.266) circle (.691);
    \draw[color=cyan, very thick] (.479,-.191) circle (.915);
\end{tikzpicture}
\caption{Degeneratively dual circles}\label{fig:degen}
\end{figure}
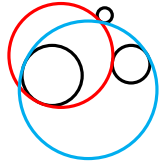

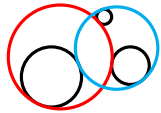
\begin{figure}
\begin{tikzpicture}
	\draw[very thick] (0,0) circle (.4);
    \draw[very thick] (1.05,.15) circle (.25);
    \draw[very thick] (.7,.8) circle (.1);
    \draw[color=red, very thick] (.117,.266) circle (.691);
    \draw[color=cyan, very thick] (.866,.384) circle (.548);
\end{tikzpicture}
\caption{Inversively dual circles}\label{fig:inversive}
\end{figure}

In this section, we will show that each of these involutions exist over any field of characteristic not 2. We will also give a generalization of external and internal tangency over such fields. Finally, we will describe a conjectural approach for generating new interpretations of $\ind_q\sigma$ via these involutions. We assume throughout this section that the cones $Q(p_i)$ meet transversely, so that the circles simultaneously tangent to $C(p_1),C(p_2),C(p_3)$ are geometrically distinct (that is, distinct over $\overline{k}$).

\subsection{Solving for tangent circles}
In order to define inversively and degeneratively dual circles over an arbitrary field $k$ with $\Char{k}\neq 2$, we would like an algebraic description of the set of circles of Apollonius. This description was provided by Coaklay~\cite{Coa60} and, independently, Stoll~\cite{Sto73}. We will describe these solutions, and we give a code implementation in~\cite{code}. We also note that while Coaklay only worked over $\mb{R}$, the solutions are in fact valid over any field in which the quadratic formula holds (i.e. $\Char{k}\neq 2$).

Let $k$ be a field with $\Char{k}\neq 2$. For $1\leq i\leq 3$, let $a_i,b_i,r_i^2\in k$ and $p_i=\pcoor{1:-2a_i:-2b_i:a_i^2+b_i^2-r_i^2}$, so that $C(p_i)\in\M$ is the $k$-rational circle with center $\pcoor{a_i:b_i:1}$ and radius squared $r_i^2$. Let $s=(s_1,s_2,s_3)\in\{1,-1\}^3$. We first define
\begin{align*}
\Delta&=\det\begin{pmatrix}a_2-a_1 & a_3-a_1\\ b_2-b_1 & b_3-b_1\end{pmatrix}\\
&=(a_1-a_2)(b_1-b_3)-(a_1-a_3)(b_1-b_2)
\end{align*}
and $D_{ij}=a_i^2-a_j^2+b_i^2-b_j^2-(r_i^2-r_j^2)$. 

\begin{rem}
Note that $\Delta\neq 0$ if and only if the the three centers $\pcoor{a_i:b_i:1}$ are not collinear.
\end{rem}

Next, let $r_i$ be a square root of $r_i^2$, and define
\begin{align}\label{eq:A,B,M,N}
A_1(s)&=\frac{(s_1r_1-s_2r_2)(b_1-b_3)-(s_1r_1-s_3r_3)(b_1-b_2)}{\Delta},\\
B_1(s)&=\frac{(s_1r_1-s_3r_3)(a_1-a_2)-(s_1r_1-s_2r_2)(a_1-a_3)}{\Delta},\nonumber\\
A_2(s)&=\frac{(b_1-b_3)D_{12}-(b_1-b_2)D_{13}}{2\Delta},\nonumber\\
B_2(s)&=\frac{(a_1-a_2)D_{13}-(a_1-a_3)D_{12}}{2\Delta},\nonumber\\
M(s)&=A_1(s)s_1r_1+A_2(s)-a_1,\nonumber\\
N(s)&=B_1(s)s_1r_1+B_2(s)-b_1.\nonumber
\end{align}
Finally, let
\begin{align}\label{eq:coaklay}
f_s(x)&=(1-A_1(s)^2-B_1(s)^2)(x-s_1r_1)^2\\
&-2(M(s)A_1(s)+N(s)B_1(s))(x-s_1r_1)\nonumber\\
&-M(s)^2-N(s)^2.\nonumber
\end{align}

\begin{rem}
If $k$ is an ordered field, we can specify that $r_i$ should be non-negative. In general, we cannot consistently choose a ``preferred'' square root of $r_i^2$. However, once we have picked $r_i$, the other square root $-r_i$ will be accounted for by negating $s_i$ in $f_s(x)$.
\end{rem}

\begin{thm}[Coaklay]\label{thm:coaklay}
The circle $C(\pcoor{1:-2\alpha_s:-2\beta_s:\alpha_s^2+\beta_s^2-\rho_s^2})$ is tangent to $C(p_1),C(p_2),C(p_3)$, where $\rho_s$ is a root of $f_s(x)$ and
\begin{align*}
\alpha_s&=A_1(s)\rho_s+A_2(s),\\
\beta_s&=B_1(s)\rho_s+B_2(s).
\end{align*}
Moreover, every circle tangent to $C(p_1),C(p_2),C(p_3)$ is obtained in this manner for some $s\in\{1,-1\}^3$.
\end{thm}

\begin{rem}
The roots of $f_s$ and $f_{-s}$ coincide, as do the sets $\{(\alpha_s,\beta_s):\rho_s\text{ a root of }f_s\}$ and $\{(\alpha_{-s},\beta_{-s}):\rho_{-s}\text{ a root of }f_{-s}\}$. In particular, we can recover all circles of Apollonius with the polynomials $f_{(1,\pm 1,\pm 1)}$.
\end{rem}

\begin{rem}\label{rem:coaklay distinct roots}
The polynomials $f_s$ have two distinct roots unless (i) the discriminant of $f_s$ vanishes, or (ii) $1-A_1(s)^2-B_1(s)^2=0$. Case (i) corresponds to having tangent circles of multiplicity 2 (compare with Proposition~\ref{prop:transverse cones}). Case (ii) corresponds to one of the tangent circles having infinite radius (i.e. lying on the plane $C(\pcoor{0:x_1:x_2:x_3})\subset\M$ at infinity), which occurs precisely when $C(p_1),C(p_2),C(p_3)$ share a common tangent line (which constitutes a component of the degenerate tangent circle). While we have largely ignored tangent circles of infinite radius to simplify our geometric considerations, such circles are still valid, algebraic solutions to the problem of Apollonius. It follows that Coaklay's equations yield distinct tangent circles unless two of $C(p_1),C(p_2),C(p_3)$ are tangent.
\end{rem}

Over $\mb{R}$, Coaklay remarks that the signs $s$ correspond to the tangency directions of the circles determined by $f_s$ relative to the $C(p_i)$. For example, one of the two circles determined by $f_{(1,1,-1)}$ will meet $C(p_1)$ and $C(p_2)$ externally and $C(p_3)$ internally, while the other circle will meet $C(p_1)$ and $C(p_2)$ internally and $C(p_3)$ externally. In particular, the two circles determined by $f_s$ are inversively dual over $\mb{R}$. Over more general fields, we will define two circles of Apollonius to be inversively dual if they are both determined by $f_s$.

\begin{defn}\label{def:inversively dual}
Let $k$ be a field with $\Char{k}\neq 2$. For each $(s_2,s_3)\in\{1,-1\}^2$, the circles tangent to $C(p_1),C(p_2),C(p_3)$ corresponding to the roots of $f_{(1,s_2,s_3)}$ are said to be \textit{inversively dual} to one another.
\end{defn}

\subsection{Dual circles via degeneration}
As we degenerate the circle $C(p_i)$ to a point, the circles of Apollonius carve out a family of circles. Using Coaklay's solutions, we will show that this family of circles of Apollonius is finite and flat. The key will be the quadratic formula over $k[t]$, which is well-defined due to our assumption that $\Char{k}\neq 2$.

In order to degenerate the circle $C(p_i)$ to a point, we want a family of squared radii that interpolate between $r_i^2$ and 0. We will use the family $t^2r_i^2$. We start by modifying Coaklay's equations to handle this degenerating family.

\begin{defn}\label{defn:coaklay family}
Let $t$ be an indeterminate over $k$. Define $A^i_1(s,t)$, $B^i_1(s,t)$, $M^i(s,t)$, and $N^i(s,t)$ by replacing $r_i$ with $tr_i$ in $A_1(s)$, $B_1(s)$, $M(s)$, and $N(s)$, respectively (see Equation~\ref{eq:A,B,M,N}). Let $R^i_1$ be $r_1$ if $i\neq 1$ and $tr_1$ if $i=1$.
\begin{align*}
f^i_{s,t}(x)&=(1-A^i_1(s,t)^2-B^i_1(s,t)^2)(x-s_1R^i_1)^2\\
&-2(M^i(s,t)A^i_1(s,t)+N^i(s,t)B^i_1(s,t))(x-s_1R^i_1)\\
&-M^i(s,t)^2-N^i(s,t)^2.
\end{align*}
Given a root $\rho^i_{s,t}$ of $f^i_{s,t}(x)$ (which we can solve for using the quadratic formula for polynomials over $k[t]$ with $\Char{k}\neq 2$), set
\begin{align*}
\alpha^i_{s,t}&=A^i_1(s,t)\rho^i_{s,t}+A_2(s),\\
\beta^i_{s,t}&=B^i_1(s,t)\rho^i_{s,t}+B_2(s).
\end{align*}
\end{defn}

By construction, the equations given in Definition~\ref{defn:coaklay family} give a parameterization of the family of circles tangent to $C(p_1),C(p_2),C(p_3)$ as $C(p_i)$ degenerates to a point. Geometrically, this family of circles arises from the family of intersections of the cones $Q(p_1),Q(p_2),Q(p_3)$ as $Q(p_i)$ degenerates to a double plane:

\begin{defn}
For $1\leq j\leq 3$, let
\begin{align*}
Q^i_t(p_j)=\begin{cases}
\mb{V}((a_iX_i+b_iY_i+Z'_i)^2-t^2r_i^2(X_i^2+Y_i^2)) & j=i,\\
\mb{V}((a_jX_j+b_jY_j+Z_j)^2-r_j^2(X_j^2+Y_j^2)) & j\neq i,
\end{cases}
\end{align*}
where $X_\ell=c_1+2a_\ell c_0$, $Y_\ell=c_2+2b_\ell c_0$, $Z_\ell=c_3+(r_\ell^2-a_\ell^2-b_\ell^2)c_0$, and $Z'_\ell=c_3+(t^2r_\ell^2-a_\ell^2-b_\ell^2)c_0$ (see Lemma~\ref{lem:cone}). By construction, we have $Q^i_t(p_j)=Q(p_j)$ for $i\neq j$. We also have $Q^i_1(p_i)=Q(p_i)$ and $Q^i_0(p_i)_\text{red}=V(p_i)$ (see Remark~\ref{rem:doubled}).
\end{defn}

As $C(p_i)$ degenerates, the family of circles of Apollonius is given by $\mc{A}^i:=\bigcap_{j=1}^3 Q^i_t(p_j)\to\Spec{k[t]}$. Ideally, we would like to define degenerative duality geometrically in terms of the fiber $\mc{A}^i_1$. Unfortunately, since $\mc{A}^i$ branches at multiple points along $\mb{A}^1_k$ (see Lemma~\ref{lem:ramification}), we cannot effectively keep track of which points in the fiber $\mc{A}^i_1$ coincide as we pass to $\mc{A}^i_0$. Instead, we will apply Coaklay's equations over $k[t]$ (Definition~\ref{defn:coaklay family}) to define degenerative duality, which will allow us to distinguish between the fibers of $\mc{A}^i$ as we pass through the ramification locus. We give a code implementation of degenerative duality in Appendix~\ref{sec:degen code}.

\begin{defn}\label{defn:degenerative}
Let $s=(1,s_1,s_2)$ and $s'=(1,s_1',s_2')$. Given roots $\rho^i_{s,t}$ and $\rho^i_{s',t}$ of $f^i_{s,t}$ and $f^i_{s',t}$, respectively, let
\begin{align*}
\alpha^i_{s,t}&=A^i_1(s,t)\rho^i_{s,t}+A_2(s),\\
\alpha^i_{s',t}&=A^i_1(s',t)\rho^i_{s',t}+A_2(s'),\\
\beta^i_{s,t}&=B^i_1(s,t)\rho^i_{s,t}+B_2(s),\\
\beta^i_{s',t}&=B^i_1(s',t)\rho^i_{s',t}+B_2(s'),\\
C_t&=C(\pcoor{1:-2\alpha^i_{s,t}:-2\beta^i_{s,t}:(\alpha^i_{s,t})^2+(\beta^i_{s,t})^2-(\rho^i_{s,t})^2}),\\
C_t'&=C(\pcoor{1:-2\alpha^i_{s',t}:-2\beta^i_{s',t}:(\alpha^i_{s',t})^2+(\beta^i_{s',t})^2-(\rho^i_{s',t})^2}).
\end{align*}
We say that two tangent circles $C(q),C(q')$ are \textit{degeneratively dual through $C(p_i)$} if there exist $s$ and $s'$ such that there is a root $\rho^i_{s,t}$ of $f^i_{s,t}$ and a root $\rho^i_{s',t}$ of $f^i_{s',t}$ satisfying $C_1=C(q)$, $C_1'=C(q')$, and $C_0=C_0'$. We denote $C(q)$ that is degeneratively dual to $C(q')$ through $C(p_i)$ by writing $\vartheta_i C(q)=C(q')$ or simply $\vartheta_i q=q'$.
\end{defn}

While the geometric picture of degenerative duality is clear, its algebraic analog in Definition~\ref{defn:degenerative} is somewhat unwieldy. We elucidate the structure of these degenerative dualities by representing each $\vartheta_i$ as a permutation matrix acting on the set $\mc{A}$ of circles of Apollonius.

\begin{prop}\label{prop:degen matrices}
Let $\bm{e}_n$ be the $n\textsuperscript{th}$ standard basis vector in $\mb{F}_2^8$. Represent the circles of Apollonius by the set $\{\bm{e}_n\}_{n=1}^8$, subject to the requirement that the subsets $\{\bm{e}_n,\bm{e}_{n+1}\}$ each correspond to the inversively dual pair of circles determined by $f_s$, where $s$ and $n$ are as in Table~\ref{table:signs}.
\begin{table}[h]
\caption{Signs and corresponding vectors}\label{table:signs}
\renewcommand{\arraystretch}{1.6}
\aboverulesep=0ex
\belowrulesep=0ex
\centering
\begin{tabular}{l|c}
\toprule
\multicolumn{1}{c}{$s$} & \phantom{m}$n$\phantom{m}\\
\toprule
$(1,1,1)$ & 1\\
$(1,1,-1)$ & 3\\
$(1,-1,1)$ & 5\\
$(1,-1,-1)$ & 7\\
\bottomrule
\end{tabular}
\end{table}
Then the action of $\vartheta_i$ on $\mc{A}$ is given by one of
\begin{align*}
\begin{psmallmatrix}
&&&&&&& 1\\
&&&&&& 1 &\\
&&&&& 1 &&\\
&&&& 1 &&&\\
&&& 1 &&&&\\
&& 1 &&&&&\\
& 1 &&&&&&\\
1 &&&&&&&\\
\end{psmallmatrix},\quad
\begin{psmallmatrix}
&&&& 1 &&&\\
&&&&& 1 &&\\
&&&&&& 1 &\\
&&&&&&& 1\\
1 &&&&&&&\\
& 1 &&&&&&\\
&& 1 &&&&&\\
&&& 1 &&&&\\
\end{psmallmatrix},\quad
\begin{psmallmatrix}
&& 1 &&&&&\\
&&& 1 &&&&\\
1 &&&&&&&\\
& 1 &&&&&&\\
&&&&&& 1 &\\
&&&&&&& 1\\
&&&& 1 &&&\\
&&&&& 1 &&\\
\end{psmallmatrix}.
\end{align*}
Moreover, the matrix representations of $\vartheta_i$ and $\vartheta_j$ are distinct for $i\neq j$.
\end{prop}
\begin{proof}
The proof is given via symbolic computation in Appendix~\ref{sec:degen code}. The code given in this appendix takes indeterminates representing the centers and (choices of) radii of three circles. For $i=1,2,3$, the code then applies the equations of Definition~\ref{defn:coaklay family} to compute whether two circles are degeneratively dual through $C(p_i)$. The $(m,n)\textsuperscript{th}$ entry of the matrix for $\vartheta_i$ consists of a 1 if $\bm{e}_m$ and $\bm{e}_n$ are degeneratively dual through $C(p_i)$ and a 0 otherwise.
\end{proof}

In the following lemma, we investigate the ramification locus of $\mc{A}^i\to\Spec{k[t]}$.

\begin{lem}\label{lem:ramification}
Let $\mc{A}^i:=\bigcap_{j=1}^3 Q^i_t(p_j)\to\Spec{k[t]}$ be the family of circles of Apollonius as the radius of $C(p_i)$ degenerates. Assume that no two circles among $C(p_1),C(p_2),C(p_3)$ are tangent. Then $\mc{A}^i\to\Spec{k[t]}$ is unramified away from a finite set of closed points. Moreover, $\mc{A}^i\to\Spec{k[t]}$ branches at some $t\neq 0$ if and only if a circle in the fiber $\mc{A}^i_t$ is inversively dual to itself.
\end{lem}
\begin{proof}
By Proposition~\ref{prop:transverse cones}, our assumption that no two circles among $C(p_1),C(p_2),C(p_3)$ are tangent ensures that the cones $Q(p_1),Q(p_2),Q(p_3)$ meet transversely over $\overline{k}$. In particular, the family $\mc{A}^i\to\Spec{k[t]}$ is unramified at $t=1$. By the same reasoning, $\mc{A}^i\to\Spec{k[t]}$ is unramified at $t\neq 0$ whenever $C(\pcoor{1:-2a_i:-2b_i:a_i^2+b_i^2-t^2r_i^2})$ is not tangent to $C(p_j)$ for $i\neq j$. 

Our goal is now to understand the ramification locus. The fiber above $t=0$ consists of four double points (compare with Section~\ref{sec:CCP}), which correspond to the confluence of the four pairs of degeneratively dual circles. The family $\mc{A}^i\to\Spec{k[t]}$ branches at $t\neq 0$ if and only if $C(\pcoor{1:-2a_i:-2b_i:a_i^2+b_i^2-t^2r_i^2})$ is tangent to $C(p_j)$ for some $j\neq i$. In terms of Coaklay's equations, this ramification happens only if the discriminant of some $f^i_{s,t}$ vanishes (see Remark~\ref{rem:coaklay distinct roots}). That is, the ramification of $\mc{A}^i$ away from $t=0$ consists solely of circles coinciding with their inversive duals. Moreover, this can only happen finitely often, since the discriminant of each $f^i_{s,t}$ is a polynomial in $t$.
\end{proof}

\begin{rem}\label{rem:family of circles}
Since the ramification locus of $\mc{A}^i\to\Spec{k[t]}$ is finite by Lemma~\ref{lem:ramification}, it follows from Lemma~\ref{lem:finiteness} that $\mc{A}^i\to\Spec{k[t]}$ is finite and flat. Our assumptions ensure that the circles tangent to $C(p_1),C(p_2),C(p_3)$ are geometrically distinct, so the fiber $\mc{A}^i_1$ consists of 8 geometrically distinct points. Thus $\mc{A}^i$ is (geometrically) an 8-sheeted cover of $\mb{A}^1_k$. Lemma~\ref{lem:ramification} implies that these sheets fall into two sets $S_1,S_2$, each containing four sheets, such that $\varsigma_1\cap \varsigma_2=\varnothing$ for each $\varsigma_1\in S_1$ and $\varsigma_2\in S_2$.
\end{rem}

Based on this description, it looks like $\mc{A}^i$ has two connected components of degree four, with each component consisting of a pair of degeneratively dual circles and their inversive duals. If this were the case, then we could use the familial local degree to prove that the sum of the local indices of two inversively dual pairs of degeneratively dual circles is an even multiple of the hyperbolic form. 

The issue is that the sheets of our branched cover $\pi:\mc{A}^i\to\mb{A}^1_k$ are given by an algebraic function that does not necessarily define a morphism. Each sheet defines a set-theoretic section $\varsigma:\mb{A}^1_k\to\mc{A}^i$. If $\varsigma$ is in fact a morphism, then the fact that $\pi\circ\varsigma$ is closed and $\pi$ is separated implies that $\varsigma$ is a closed immersion. It would then follow that $\mc{A}^i$ has eight irreducible components, and Lemma~\ref{lem:ramification} would imply that $\mc{A}^i$ indeed has two connected components of degree four. In this hypothetical scenario, we can apply the familial local degree to compute the sum of local indices over the fibers in one of these connected components.

\begin{lem}\label{lem:sum of degen duals}
Let $C(q),C(q')$ be inversively dual circles tangent to $C(p_1),C(p_2),C(p_3)$. Assume that $\mc{A}^i$ consists of two connected components of degree four as detailed in Remark~\ref{rem:family of circles}. Let $Y\subset\mc{A}^i$ be the connected component to which $C(q),C(q')$ belong, and denote $Y_0=\{C(d),C(d')\}$. Assume that $k(d)\cong k(d')$, and that $k(d)/k$ is separable. Then
\[\ind_q\sigma+\ind_{q'}\sigma+\ind_{\vartheta_i q}\sigma+\ind_{\vartheta_i q'}\sigma=2\Tr_{k(d)/k}\mb{H}.\]
\end{lem}
\begin{proof}
By~\cite[Theorem 1.3]{BBMMO21}, we may assume that $k(d)=k$. Remark~\ref{rem:family of circles} implies that $\mc{A}^i$ satisfies the criteria of Theorem~\ref{thm:harder}. Remark~\ref{rem:family of circles} furthermore characterizes the fiber $Y_1=\{C(q),C(q'),\vartheta_iC(q),\vartheta_iC(q')\}$, so we have 
\[\ind_q\sigma+\ind_{q'}\sigma+\ind_{\vartheta_i q}\sigma+\ind_{\vartheta_i q'}\sigma=\ind_d\sigma+\ind_{d'}\sigma.\]
Now $\ind_d\sigma=\ind_{d'}\sigma=\mb{H}$ by Proposition~\ref{prop:CCP gives H}.
\end{proof}

\subsection{Cube of Apollonius}
Our next goal is to show how Lemma~\ref{lem:sum of degen duals}, paired with the convenient symmetry of the circles of Apollonius, would allow us to relate the local indices of inversively dual circles. We begin by studying the symmetry of the circles of Apollonius.

Let $\mc{A}$ denote the set of circles tangent to $C(p_1),C(p_2),C(p_3)$. Note that $\mc{A}=\mc{A}^i_1$ for each $i$. Degenerative duality gives us three involutions $\vartheta_1,\vartheta_2,\vartheta_3:\mc{A}\to\mc{A}$. We will show that the eight elements of $\mc{A}$ correspond to the vertices of a cube such that each $\vartheta_i$ is a reflection of the cube across a central plane parallel to one of the faces. We illustrate this in Figure~\ref{fig:cube}, with the relevant circles labeled in Figure~\ref{fig:circle labels}.

\begin{lem}\label{lem:cube of apollonius}
There exists a cube with vertices $\mc{A}$ such that each $\vartheta_i$ is a reflection of the cube across a central plane parallel to one of its faces.
\end{lem}
\begin{proof}
The cubical graph is the unique connected bipartite trivalent graph. Let $G$ be the graph whose set of vertices is $\mc{A}$, and where two elements of $\mc{A}$ share an edge if and only if they are degeneratively dual to one another via some $\vartheta_i$. We will prove the following by appealing to Proposition~\ref{prop:degen matrices}:
\begin{enumerate}[(i)]
\item If $C\in\mc{A}$, then $\vartheta_i(C)\neq C$ for all $1\leq i\leq 3$.

\textit{Proof.} Each $\vartheta_i$ does not fix any standard basis vector.

\item If $C\in\mc{A}$, then $\vartheta_i(C)\neq\vartheta_j(C)$ for $i\neq j$.

\textit{Proof.} For each $1\leq n\leq 8$, the vectors $\vartheta_i\bm{e}_n$ and $\vartheta_j\bm{e}_n$ are distinct.

\item If $C\in\mc{A}$, then $\vartheta_i\vartheta_j(C)=\vartheta_j\vartheta_i(C)$.

\textit{Proof.} Each pair of the three matrices in Proposition~\ref{prop:degen matrices} commute.

\item If $C\in\mc{A}$, then $\vartheta_1\vartheta_2\vartheta_3(C)\neq C$.

\textit{Proof.} As a matrix, we have
\[\vartheta_1\vartheta_2\vartheta_3=\begin{psmallmatrix}
& 1 &&&&&&\\
1 &&&&&&&\\
&&& 1 &&&&\\
&& 1 &&&&&\\
&&&&& 1 &&\\
&&&& 1 &&&\\
&&&&&&& 1\\
&&&&&& 1 &
\end{psmallmatrix}.\]
This does not fix $\bm{e}_n$ for any $1\leq n\leq 8$.
\end{enumerate}
Item (i) states that $G$ contains no loops, so that $G$ is indeed a graph. It follows from (ii) that $G$ is trivalent. Note that $\vartheta_i^2=\id_\mc{A}$ for all $i$, so (ii) and (iii) imply that any odd cycle in $G$ must be a 3-cycle obtained by applying $\vartheta_1,\vartheta_2,\vartheta_3$ in any order. It then follows from (iv) that $G$ contains no odd cycles, so $G$ is bipartite. Items (i) through (iv) imply that there are 8 operations $\{\vartheta_1^i\vartheta_2^j\vartheta_3^\ell\}_{i,j,\ell=0}^1$ that all yield distinct circles; since $\mc{A}$ has 8 vertices, this implies that $G$ is connected and is therefore the cubical graph.

Now thinking of $G$ as a cube, item (iii) states that each face of $G$ is given by two instances of $\vartheta_i$ on a pair of parallel edges and two instances of $\vartheta_j$ on the remaining pair of parallel edges. This in turn implies that all four instances of $\vartheta_i$ form a set of four parallel edges on $G$, and $\vartheta_i$ swaps the faces of $G$ that do not contain any of these four parallel edges.
\end{proof}

\def\ooo{\begin{tikzpicture}[scale=.75]
	\draw[very thick] (0,0) circle (.4);
    \draw[very thick] (1.05,.15) circle (.25);
    \draw[very thick] (.7,.8) circle (.1);
    \draw[color=red, very thick] (.566,.419) circle (.304);
\end{tikzpicture}}
\def\ioo{\begin{tikzpicture}[scale=.75]
	\draw[very thick] (0,0) circle (.4);
    \draw[very thick] (1.05,.15) circle (.25);
    \draw[very thick] (.7,.8) circle (.1);
    \draw[color=red, very thick] (.117,.266) circle (.691);
\end{tikzpicture}}
\def\oio{\begin{tikzpicture}[scale=.75]
	\draw[very thick] (0,0) circle (.4);
    \draw[very thick] (1.05,.15) circle (.25);
    \draw[very thick] (.7,.8) circle (.1);
    \draw[color=red, very thick] (.558,.552) circle (.385);
\end{tikzpicture}}
\def\ooi{\begin{tikzpicture}[scale=.75]
	\draw[very thick] (0,0) circle (.4);
    \draw[very thick] (1.05,.15) circle (.25);
    \draw[very thick] (.7,.8) circle (.1);
    \draw[color=red, very thick] (.843,.241) circle (.477);
\end{tikzpicture}}
\def\iio{\begin{tikzpicture}[scale=.75]
	\draw[very thick] (0,0) circle (.4);
    \draw[very thick] (1.05,.15) circle (.25);
    \draw[very thick] (.7,.8) circle (.1);
    \draw[color=red, very thick] (-.052,.505) circle (.908);
\end{tikzpicture}}
\def\ioi{\begin{tikzpicture}[scale=.75]
	\draw[very thick] (0,0) circle (.4);
    \draw[very thick] (1.05,.15) circle (.25);
    \draw[very thick] (.7,.8) circle (.1);
    \draw[color=red, very thick] (.479,-.191) circle (.915);
\end{tikzpicture}}
\def\oii{\begin{tikzpicture}[scale=.75]
	\draw[very thick] (0,0) circle (.4);
    \draw[very thick] (1.05,.15) circle (.25);
    \draw[very thick] (.7,.8) circle (.1);
    \draw[color=red, very thick] (.866,.384) circle (.548);
\end{tikzpicture}}
\def\iii{\begin{tikzpicture}[scale=.75]
	\draw[very thick] (0,0) circle (.4);
    \draw[very thick] (1.05,.15) circle (.25);
    \draw[very thick] (.7,.8) circle (.1);
    \draw[color=red, very thick] (.447,.088) circle (.856);
\end{tikzpicture}}

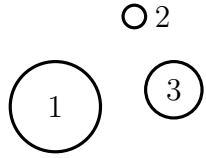
\begin{figure}
\begin{tikzpicture}[scale=1.5]
	\draw[very thick] (0,0) circle (.4);
    \draw[very thick] (1.05,.15) circle (.25);
    \draw[very thick] (.7,.8) circle (.1);
    \node at (0,0) {$1$};
    \node at (1.05,.15) {$3$};
    \node at (.95,.8) {$2$};
\end{tikzpicture}
\caption{Circle labels}
\label{fig:circle labels}
\end{figure}

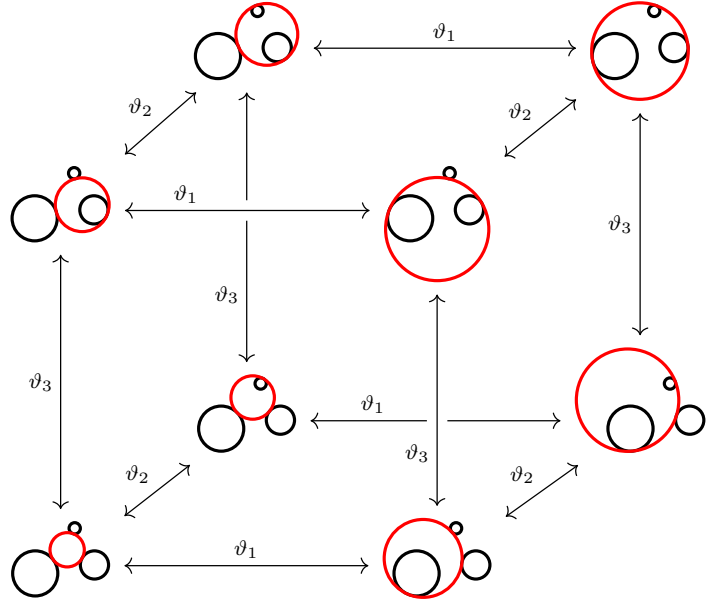
\begin{figure}
\begin{tikzcd}[row sep=scriptsize, column sep=scriptsize,every arrow/.append style={leftrightarrow}]
& \oii \arrow[rr, "\vartheta_1"] \arrow[from=dd, "\vartheta_3" near start] & & \iii \\
\ooi \arrow[ur, "\vartheta_2"] \arrow[rr, crossing over, "\vartheta_1" near start] & & \ioi \arrow[ur, "\vartheta_2"] \\
& \oio \arrow[rr, "\vartheta_1" near start] & & \iio \arrow[uu, "\vartheta_3"] \\
\ooo \arrow[ur, "\vartheta_2"] \arrow[rr, "\vartheta_1"] \arrow[uu, "\vartheta_3"] & & \ioo \arrow[uu, crossing over, "\vartheta_3" near start] \arrow[ur, "\vartheta_2"]\\
\end{tikzcd}
\caption{Cube of Apollonius}
\label{fig:cube}
\end{figure}

\begin{rem}
Phrased differently, Lemma~\ref{lem:cube of apollonius} states that the set of circles of Apollonius is a $(\mb{Z}/2\mb{Z})^3$-torsor.
\end{rem}

Over the reals, $\vartheta_i$ swaps internal and external tangency with respect to $C(p_i)$. Since inversive duality swaps internal and external tangency with respect to all three original circles, we see that inversive duality is equal to $\vartheta_1\vartheta_2\vartheta_3$. In terms of Figure~\ref{fig:cube}, inversively dual circles are body-diagonal in the cube of Apollonius. It turns out that this still holds over any field of characteristic not 2:

\begin{prop}\label{prop:inversive dual}
The circles $C$ and $\vartheta_1\vartheta_2\vartheta_3(C)$ are inversively dual for each $C\in\mc{A}$.
\end{prop}
\begin{proof}
In the context of Proposition~\ref{prop:degen matrices}, we have
\[\vartheta_1\vartheta_2\vartheta_3=\begin{psmallmatrix}
& 1 &&&&&&\\
1 &&&&&&&\\
&&& 1 &&&&\\
&& 1 &&&&&\\
&&&&& 1 &&\\
&&&& 1 &&&\\
&&&&&&& 1\\
&&&&&& 1 &
\end{psmallmatrix}.\]
By assumption, $\{\bm{e}_i,\bm{e}_{i+1}\}$ forms a pair of inversively dual circles for each $i=1,3,5,7$, so it follows that $\vartheta_1\vartheta_2\vartheta_3$ swaps inversively dual circles.
\end{proof}

We can now show that the sum of local indices of inversively dual circles is hyperbolic.

\begin{lem}\label{lem:inversive sum to h}
Let $C(q),C(q')\in\mc{A}$ be inversively dual to one another. Assume that $C(q)\neq C(q')$. Further assume that $k(q)\cong k(q')$, that $k(q)/k$ is separable, and that for each $1\leq i\leq 3$, the extension $k(d)/k$ is separable for each double point $d\in\mc{A}^i_0$. Then
\[\ind_q\sigma+\ind_{q'}\sigma=\Tr_{k(q)/k}\mb{H}.\]
\end{lem}
\begin{proof}
By the running assumption that the circles of Apollonius $\mc{A}$ are geometrically distinct, together with our assumption that $k(q)\cong k(q')$, the form $\ind_q\sigma+\ind_{q'}\sigma$ has rank $2[k(q):k]$. It thus suffices to prove that $\ind_q\sigma+\ind_{q'}\sigma$ is hyperbolic. We will actually show that $2(\ind_q\sigma+\ind_{q'}\sigma)$ is an even multiple of a hyperbolic form, which will also suffice.

By Lemma~\ref{lem:sum of degen duals}, the sum of local indices of the four circles on any two body diagonal edges of the cube of Apollonius is an even multiple of the hyperbolic form (see Figure~\ref{fig:body diagonal}). By adding or subtracting local indices along three edges and their body diagonals, we can express $2(\ind_q\sigma+\ind_{q'}\sigma)$ in terms amenable to Lemma~\ref{lem:sum of degen duals} (see Figure~\ref{fig:sum hyperbolic}). Explicitly,
\begin{align*}
2(\ind_q\sigma+\ind_{q'}\sigma)&=2(\ind_q\sigma+\ind_{\vartheta_1\vartheta_2\vartheta_3q}\sigma)\\
&=\ind_q\sigma+\ind_{\vartheta_1q}\sigma-(\ind_{\vartheta_1q}\sigma+\ind_{\vartheta_1\vartheta_3q}\sigma)+\ind_{\vartheta_1\vartheta_3q}\sigma+\ind_{\vartheta_1\vartheta_2\vartheta_3q}\sigma\\
&+\ind_q\sigma+\ind_{\vartheta_2q}\sigma-(\ind_{\vartheta_2q}\sigma+\ind_{\vartheta_2\vartheta_3q}\sigma)+\ind_{\vartheta_2\vartheta_3q}\sigma+\ind_{\vartheta_1\vartheta_2\vartheta_3q}\sigma\\
&=(\ind_q\sigma+\ind_{\vartheta_1\vartheta_2\vartheta_3q}\sigma+\ind_{\vartheta_1q}\sigma+\ind_{\vartheta_2\vartheta_3q}\sigma)\\
&-(\ind_{\vartheta_1q}\sigma+\ind_{\vartheta_2\vartheta_3q}\sigma+\ind_{\vartheta_2q}\sigma+\ind_{\vartheta_1\vartheta_3q}\sigma)\\
&+(\ind_q\sigma+\ind_{\vartheta_1\vartheta_2\vartheta_3q}\sigma+\ind_{\vartheta_2q}\sigma+\ind_{\vartheta_1\vartheta_3q}\sigma).
\end{align*}
We conclude by noting that $\vartheta_2\vartheta_3q$ is the inversive dual of $\vartheta_1q$, and $\vartheta_1\vartheta_3q$ is the inversive dual of $\vartheta_2q$, so Lemma~\ref{lem:sum of degen duals} implies that $2(\ind_q\sigma+\ind_{q'}\sigma)$ is an even multiple of $\mb{H}$.
\end{proof}

\def\ooif{\begin{tikzpicture}[scale=.75,opacity=.3]
	\draw[very thick] (0,0) circle (.4);
    \draw[very thick] (1.05,.15) circle (.25);
    \draw[very thick] (.7,.8) circle (.1);
    \draw[color=red, very thick] (.843,.241) circle (.477);
\end{tikzpicture}}
\def\iiof{\begin{tikzpicture}[scale=.75,opacity=.3]
	\draw[very thick] (0,0) circle (.4);
    \draw[very thick] (1.05,.15) circle (.25);
    \draw[very thick] (.7,.8) circle (.1);
    \draw[color=red, very thick] (-.052,.505) circle (.908);
\end{tikzpicture}}
\def\ioif{\begin{tikzpicture}[scale=.75,opacity=.3]
	\draw[very thick] (0,0) circle (.4);
    \draw[very thick] (1.05,.15) circle (.25);
    \draw[very thick] (.7,.8) circle (.1);
    \draw[color=red, very thick] (.479,-.191) circle (.915);
\end{tikzpicture}}
\def\oiof{\begin{tikzpicture}[scale=.75,opacity=.3]
	\draw[very thick] (0,0) circle (.4);
    \draw[very thick] (1.05,.15) circle (.25);
    \draw[very thick] (.7,.8) circle (.1);
    \draw[color=red, very thick] (.558,.552) circle (.385);
\end{tikzpicture}}

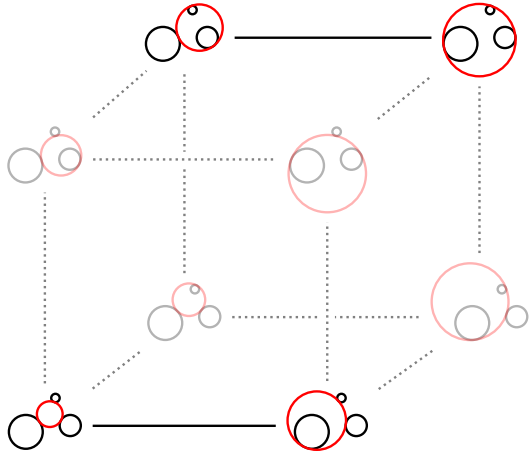
\begin{figure}
\adjustbox{scale=.75,center}{
\begin{tikzcd}[row sep=scriptsize, column sep=scriptsize,every arrow/.append style={dash, very thick}]
& \oii \arrow[rr] \arrow[from=dd, gray, dotted] & & \iii \\
\ooif \arrow[ur, gray, dotted] \arrow[rr, crossing over, gray, dotted] & & \ioif \arrow[ur, gray, dotted] \\
& \oiof \arrow[rr, gray, dotted] & & \iiof \arrow[uu, gray, dotted] \\
\ooo \arrow[ur, gray, dotted] \arrow[rr] \arrow[uu, gray, dotted] & & \ioo \arrow[uu, crossing over, gray, dotted] \arrow[ur, gray, dotted]\\
\end{tikzcd}}
\caption{Four body diagonal circles}
\label{fig:body diagonal}
\end{figure}

\begin{figure}
\adjustbox{scale=.75,center}{
\begin{tikzcd}[row sep=scriptsize, column sep=scriptsize,every arrow/.append style={dash, very thick}]
& \oii \arrow[rr, "+"scale=2] \arrow[from=dd, "-"scale=2, near start] & & \iii \\
\ooif \arrow[ur, gray, dotted] \arrow[rr, crossing over, gray, dotted] & & \ioi \arrow[ur, "+"scale=2] \\
& \oio \arrow[rr, gray, dotted] & & \iiof \arrow[uu, gray, dotted] \\
\ooo \arrow[ur, "+"scale=2] \arrow[rr, "+"scale=2] \arrow[uu, gray, dotted] & & \ioo \arrow[uu, crossing over, "-"scale=2, near end] \arrow[ur, gray, dotted]\\
\end{tikzcd}}
\caption{Sum of local indices of conjugate circles}
\label{fig:sum hyperbolic}
\end{figure}
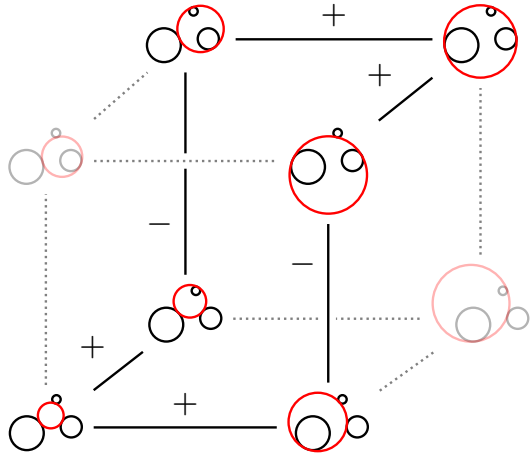

\begin{rem}\label{rem:program}
As a consequence of Lemma~\ref{lem:inversive sum to h}, we have a conditional new method for divining a geometric interpretation of $\ind_p\sigma$: any interpretation that sends inversively dual pairs to hyperbolic forms of the appropriate rank will give an enriched enumerative theorem about the circles of Apollonius.

There are two key assumptions preventing us from carrying out this program. The first is that we want our parameterization of the sheets of $\mc{A}^i\to\mb{A}^1_k$ to give us a decomposition of $\mc{A}^i$ into two connected components. The second assumption, which should be easier to resolve, is that $k(q)\cong k(q')$ for $C(q),C(q')$ inversively dual and $k(d)\cong k(d')$ for $C(d),C(d')\in\mc{A}^i_0$ the double circles to which $C(q),C(q')$ and their $i\textsuperscript{th}$ degenerative duals degenerate.
\end{rem}

\appendix
\section{Solving for the cone of tangent circles}\label{sec:sage}
We include a short piece of Sage code that performs the necessary calculation from Lemma~\ref{lem:cone}.
\begin{lstlisting}
var('x','y','z','a','b','r');
var(','.join('c%s'%i for i in range(3)));
var(','.join('A%s'%i for i in range(1,7)));
f = A1*x^2+A2*y^2+A3*z^2+A4*x*z+A5*y*z+A6*x*y;
f = f.subs(z == -a*x-b*y-2*r^2*c0);
f = expand(f.subs(x == c1+2*a*c0, y == c2+2*b*c0));
eqns = [f.coefficient(c0,2) == 4*(a^2+b^2-r^2),\
        f.coefficient(c1,2) == 1,\
        f.coefficient(c2,2) == 1,\
        f.coefficient(c0*c1,1) == 4*a,\
        f.coefficient(c0*c2,1) == 4*b,\
        f.coefficient(c1*c2,1) == 0];
solve(eqns, A1, A2, A3, A4, A5, A6)
\end{lstlisting}

\section{Degenerative duality}\label{sec:degen code}
In this appendix, we implement the equations of Definition~\ref{defn:coaklay family} to compute whether two circles are degeneratively dual through $C(p_i)$.

\begin{lstlisting}
var('x,t');
var(','.join('a%s'%i for i in range(1,4)));
var(','.join('b%s'%i for i in range(1,4)));
var(','.join('r%s'%i for i in range(1,4)));
a = [a1,a2,a3];
b = [b1,b2,b3];

def A_1(a,b,r,s):
    F = (a[0]-a[1])*(b[0]-b[2])-(a[0]-a[2])*(b[0]-b[1])
    return(((s[0]*r[0]-s[1]*r[1])*(b[0]-b[2])\
           -(s[0]*r[0]-s[2]*r[2])*(b[0]-b[1]))/F)

def B_1(a,b,r,s):
    F = (a[0]-a[1])*(b[0]-b[2])-(a[0]-a[2])*(b[0]-b[1])
    return(((s[0]*r[0]-s[2]*r[2])*(a[0]-a[1])\
           -(s[0]*r[0]-s[1]*r[1])*(a[0]-a[2]))/F)

def A_2(a,b,r,s):
    D = a[0]^2-a[1]^2+b[0]^2-b[1]^2-(r[0]^2-r[1]^2)
    E = a[0]^2-a[2]^2+b[0]^2-b[2]^2-(r[0]^2-r[2]^2)
    F = (a[0]-a[1])*(b[0]-b[2])-(a[0]-a[2])*(b[0]-b[1])
    return(((b[0]-b[2])*D-(b[0]-b[1])*E)/(2*F))

def B_2(a,b,r,s):
    D = a[0]^2-a[1]^2+b[0]^2-b[1]^2-(r[0]^2-r[1]^2)
    E = a[0]^2-a[2]^2+b[0]^2-b[2]^2-(r[0]^2-r[2]^2)
    F = (a[0]-a[1])*(b[0]-b[2])-(a[0]-a[2])*(b[0]-b[1])
    return(((a[0]-a[1])*E-(a[0]-a[2])*D)/(2*F))

def R(a,b,r,s):
    A1 = A_1(a,b,r,s)
    A2 = A_2(a,b,r,s)
    B1 = B_1(a,b,r,s)
    B2 = B_2(a,b,r,s)
    m = A2+A1*s[0]*r[0]-a[0]
    n = B2+B1*s[0]*r[0]-b[0]
    f = (x-s[0]*r[0])^2*(1-A1^2-B1^2)\
       -2*(m*A1+n*B1)*(x-s[0]*r[0])-m^2-n^2
    rts = [u.right_hand_side() for u in solve(f == 0, x)]
    return(rts)

def degen(a,b,r,signs):
    X = [];
    Y = [];
    Z = [];

    for s in signs:
        A1 = A_1(a,b,r,s);
        A2 = A_2(a,b,r,s);
        B1 = B_1(a,b,r,s);
        B2 = B_2(a,b,r,s);
        for Rad in R(a,b,r,s):
            alpha = A1*Rad+A2;
            beta  = B1*Rad+B2;
            X.append(-2*alpha);
            Y.append(-2*beta);
            Z.append(alpha^2+beta^2-Rad^2);

    W = [];
    M = [];
    for i in range(8):
        W.append([X[i].subs(t=0),\
                  Y[i].subs(t=0),\
                  Z[i].subs(t=0)]);
    for i in range(8):
        M.append([int(W[i]==W[j]) for j in range(8)]);
    return(Matrix(M)-identity_matrix(8))

signs = [];
for i in range(4):
    s = [1];
    for j in format(i,'02b'):
        s.append((-1)^int(j));
    signs.append(s);

for ii in range(3):
    r = [r1,r2,r3];
    print('degenerative duality for i =',ii+1)
    r[ii] = t*r[ii];
    print(degen(a,b,r,signs),'\n')
\end{lstlisting}

This code outputs three symmetric $8\times 8$ matrices, which we include below. Each column and row corresponds to a solution to Coaklay's equations. Columns/rows 1 and 2 correspond to the two solutions at $s=(1,1,1)$, 3 and 4 to the two solutions at $s=(1,1,-1)$, 5 and 6 to the two solutions at $s=(1,-1,1)$, and 7 and 8 to the two solutions at $s=(1,-1,-1)$. The $(m,n)^\text{th}$ entry of the matrix corresponding to the degenerative duality $\vartheta_i$ is 1 if the $m^\text{th}$ and $n^\text{th}$ solutions coincide when $r_i=0$ and is 0 otherwise. Trivially, each diagonal entry will be 1, so we subtract off the identity matrix to make the matrices for degenerative duality more readable.

\begin{lstlisting}
degenerative duality for i = 1
[0 0 0 0 0 0 0 1]
[0 0 0 0 0 0 1 0]
[0 0 0 0 0 1 0 0]
[0 0 0 0 1 0 0 0]
[0 0 0 1 0 0 0 0]
[0 0 1 0 0 0 0 0]
[0 1 0 0 0 0 0 0]
[1 0 0 0 0 0 0 0] 

degenerative duality for i = 2
[0 0 0 0 1 0 0 0]
[0 0 0 0 0 1 0 0]
[0 0 0 0 0 0 1 0]
[0 0 0 0 0 0 0 1]
[1 0 0 0 0 0 0 0]
[0 1 0 0 0 0 0 0]
[0 0 1 0 0 0 0 0]
[0 0 0 1 0 0 0 0] 

degenerative duality for i = 3
[0 0 1 0 0 0 0 0]
[0 0 0 1 0 0 0 0]
[1 0 0 0 0 0 0 0]
[0 1 0 0 0 0 0 0]
[0 0 0 0 0 0 1 0]
[0 0 0 0 0 0 0 1]
[0 0 0 0 1 0 0 0]
[0 0 0 0 0 1 0 0] 
\end{lstlisting}

\section{Local phylogeny of enriched enumerative geometry}\label{sec:phylogeny}
\newcolumntype{C}[1]{>{\centering}p{#1}}

The geometricity problem (Question~\ref{ques:geometric interpretation}) asks for a classification of enriched enumerative problems in terms of their local geometric interpretation. One issue with this question as stated is that what constitutes a ``valid'' geometric interpretation is subjective. As seen in this article, one can derive multiple intrinsically interesting local geometric descriptions for a single enriched enumerative problem. In order for the geometricity problem to become attackable, this subjectivity must be resolved.

One possible route forward is to develop not just a taxonomy, but rather a phylogeny of enriched enumerative problems. Perhaps enriched enumerative problems inherit local interpretations from their genus, family, order, and so on, with all problems belonging to the intersection-theoretic domain given by B\'ezout's theorem (see Section~\ref{sec:bezout as universal}). Expanding on Example~\ref{ex:taxa}, we will discuss a few potential phyla, which we put together in Figure~\ref{fig:phylogeny}. In the following figures, solid lines refer to established clades of problems, while dotted lines refer to conjectural relationships.

\subsection*{Rational curves on hypersurfaces}
In Example~\ref{ex:taxa} (i), we discussed how both lines on cubic surfaces and lines on quintic threefolds share a common geometric description in terms of Segre involutions~\cite{KW21,Pau20}. Conjecturally, one might expect another geometric description for lines on hypersurfaces that can also be applied to other rational curves. Levine and Pauli give various quadratic counts of twisted cubics on hypersurfaces and complete intersections (although they do not treat the local information)~\cite{LP22}, and the author has been working with Thomas Brazelton and Sabrina Pauli (BMP) to understand both the global enriched count and the local interpretation of conics on quintic threefolds. We put this phylum of ``rational curves on hypersurfaces'' together in Figure~\ref{fig:curves on surfaces}.

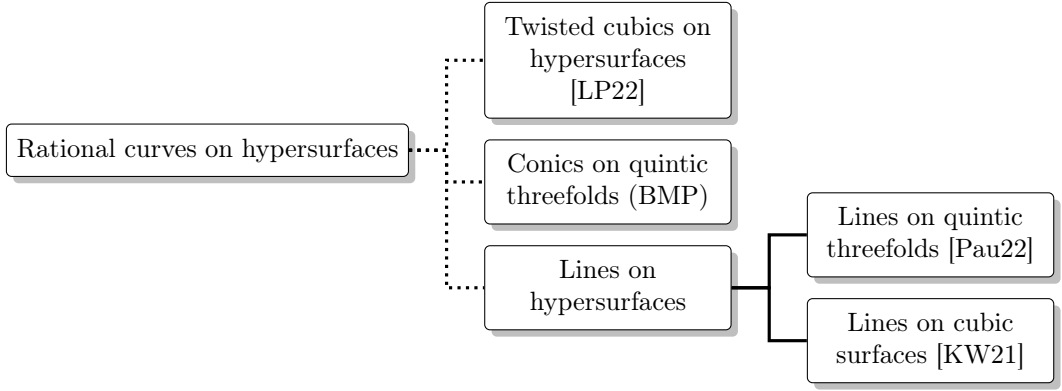
\begin{figure}
\begin{forest}
    for tree={
      if level=0{align=center}{
        align={@{}C{30mm}@{}},
      },
      grow=east,
      draw,
      font=\footnotesize,
      edge path={
        \noexpand\path [draw, \forestoption{edge}] (!u.parent anchor) -- +(5mm,0) |- (.child anchor)\forestoption{edge label};
      },
      parent anchor=east,
      child anchor=west,
      l sep=10mm,
      tier/.wrap pgfmath arg={tier #1}{level()},
      edge={very thick},
      fill=white,
      rounded corners=2pt,
      drop shadow,
    }
   [Rational curves on hypersurfaces
            [Lines on hypersurfaces,edge=dotted
                [Lines on cubic surfaces \cite{KW21}]
                [Lines on quintic threefolds \cite{Pau20}]
            ]
            [Conics on quintic threefolds (BMP),edge=dotted]
            [Twisted cubics on hypersurfaces \cite{LP22},edge=dotted]
        ]
    \end{forest}
    \caption{Rational curves on hypersurfaces}\label{fig:curves on surfaces}
\end{figure}

\subsection*{Varieties meeting a specified locus}
In Example~\ref{ex:taxa} (iii), we explained how the local geometric interpretation of conics through eight lines in $\mb{P}^3$~\cite{DGGM21} can be applied more generally to plane curves of higher degree through larger collections of lines. These counts, along with the count of twisted cubics through twelve lines, are the subject of ongoing joint work of the author and Sabrina Pauli (MP). Counting twisted cubics meeting twelve lines is a ``space curves through lines'' problem, and one might expect such problems to be closely related to ``plane curves through lines'' problems.

Thinking of points as linear varieties, the count of rational curves through a fixed number of points appears, at least superficially, to be related to counting space curves through lines. Ongoing work of Jesse Kass, Marc Levine, Jake Solomon, and Kirsten Wickelgren (KLSW) gives an enriched count of rational curves through sets of points, with the local geometric interpretation being given by an enriched Welschinger invariant (see e.g.~\cite[Section 9]{PW20}). Could enriched Welschinger invariants provide an alternative geometric description for counts of (rational) space curves through lines? We illustrate this conjectural relationship in Figure~\ref{fig:curves linear}.
\begin{figure}
\begin{forest}
    for tree={
      if level=0{align=center}{
        align={@{}C{30mm}@{}},
      },
      grow=east,
      draw,
      font=\footnotesize,
      edge path={
        \noexpand\path [draw, \forestoption{edge}] (!u.parent anchor) -- +(5mm,0) |- (.child anchor)\forestoption{edge label};
      },
      parent anchor=east,
      child anchor=west,
      l sep=10mm,
      tier/.wrap pgfmath arg={tier #1}{level()},
      edge={very thick},
      fill=white,
      rounded corners=2pt,
      drop shadow,
    }
    [Curves meeting linear spaces,edge=dotted
                [Rational curves through points (KLSW),edge=dotted]
                [Curves through lines,edge=dotted
                    [Conics through lines \cite{DGGM21}]
                    [Plane curves through lines (MP)]
                    [Space curves through lines,edge=dotted]
                ]
            ]
    \end{forest}
    \caption{Curves meeting linear spaces}\label{fig:curves linear}
\end{figure}
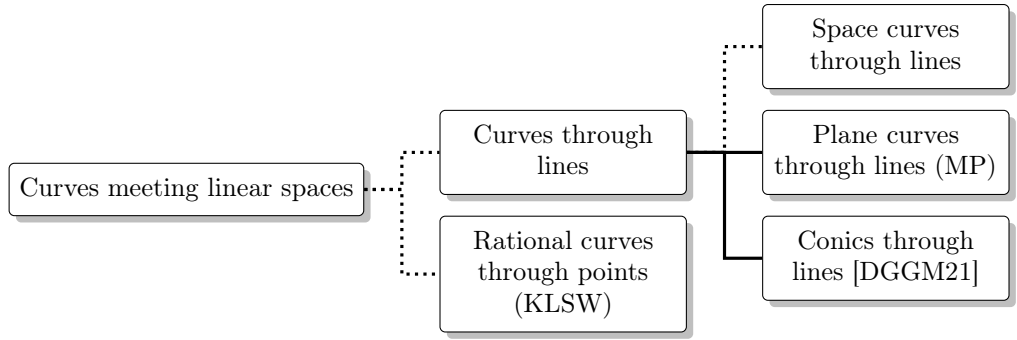

As discussed in Example~\ref{ex:taxa} (ii), Srinivasan--Wickelgren's enriched count of lines through codimension 2 planes form a family of enriched enumerative problems whose shared geometric interpretation is given by a generalization of the cross-ratio. Work of Brazelton on the Wronski problem~\cite{Bra22} shows that these same generalized cross-ratios can be used as a geometric interpretation for the count of $d$-planes meeting $(n-d)$-planes in $\mb{P}^n$. We thus obtain a family of problems of the form ``linear spaces meeting linear spaces,'' as shown in Figure~\ref{fig:varieties locus}.

Brazelton also indicates that the Wrosnki problem might have a second geometric interpretation in terms of the enriched Welschinger invariant of Kass--Levine--Solomon--Wickelgren. This would provide an intriguing connection between problems of the form ``linear spaces meeting linear spaces'' and problems of the form ``curves meeting linear spaces.'' Another potentially related result is Cotterill--Darago--Han's enriched Pl\"ucker formula for linear series on hyperelliptic curves~\cite{CDH20}. While this article does not give a geometric description for the relevant local indices, there are formulas for the local indices in terms of Wronskian determinants. This suggests that Brazelton's geometric interpretations could be applied to relate~\cite{CDH20} to the other problems listed in Figure~\ref{fig:varieties locus}.
\begin{figure}
\begin{forest}
    for tree={
      if level=0{align=center}{
        align={@{}C{30mm}@{}},
      },
      grow=east,
      draw,
      font=\footnotesize,
      edge path={
        \noexpand\path [draw, \forestoption{edge}] (!u.parent anchor) -- +(5mm,0) |- (.child anchor)\forestoption{edge label};
      },
      parent anchor=east,
      child anchor=west,
      l sep=10mm,
      tier/.wrap pgfmath arg={tier #1}{level()},
      edge={very thick},
      fill=white,
      rounded corners=2pt,
      drop shadow,
    }
    [Varieties meeting specified locus
            [Curves meeting linear spaces,edge=dotted]
            [Linear spaces meeting linear spaces \cite{Bra22},edge=dotted
                [Lines through hyperplanes \cite{SW21}]
            ]
            [Linear series on hyperelliptic curves \cite{CDH20},edge=dotted]
        ]
    \end{forest}
    \caption{Varieties meeting specified locus}\label{fig:varieties locus}
\end{figure}
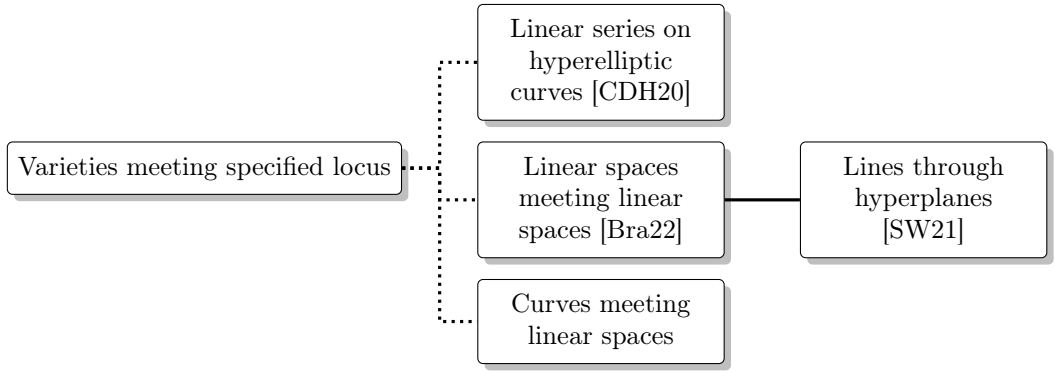

\subsection*{Tangency problems}
We now turn to problems whose local interpretation may be related to this article's treatment of the circles of Apollonius (McK22). Classically, the count of spheres tangent to four given spheres can be derived from the count of circles tangent to three given circles. It seems reasonable to expect a geometric interpretation analogous to the one given in Lemma~\ref{lem:geometric interpretation} to hold for the count of spheres tangent to four spheres. One could call this hypothetical family ``quadric tangency problems,'' since one is interested in counting quadric varieties that are tangent to a given collection of objects.

Speculatively, there should be a connection between Larson--Vogt's ``Qtype'' for bitangents to plane quartics~\cite[Definition 1.2]{LV21} and the geometric interpretations that would arise if Remark~\ref{rem:program} happens to work out. Given a circle $C(q)$ tangent to a given trio of circles $C(p_1),C(p_2),C(p_3)$, there should be an invariant $t_i(q)$ that records an enrichment of the tangency direction of $C(q)$ to $C(p_i)$. Moreover, if $C(q')$ is the inversive dual of $C(q)$, then it should hold that $\langle\prod_{i=1}^3 t_i(q)\rangle+\langle\prod_{i=1}^3 t_i(q')\rangle$ is hyperbolic. This product of tangency directions would then be analogous to the Qtype $\partial_L f(z_1)\cdot\partial_L f(z_2)$, which records the ``tangency directions'' of a bitangent $L$ to the plane quartic defined by $f$ at the two points $z_1,z_2$ of tangency. 

This suggests that while the count of bitangents to plane quartics is a ``linear tangency'' problem, there is a larger family of ``tangency'' problems encompassing both of these results (see Figure~\ref{fig:tangency problems}). We posit that such tangency problems can be characterized by local geometric interpretations that are products of ``tangency directions'' over the locus of tangency. This would be compelling evidence in favor of a phylogenetic approach to classifying enriched enumerative problems, since these shared geometric interpretations would arise from completely different calculations.

\begin{figure}
\begin{forest}
    for tree={
      if level=0{align=center}{
        align={@{}C{30mm}@{}},
      },
      grow=east,
      draw,
      font=\footnotesize,
      edge path={
        \noexpand\path [draw, \forestoption{edge}] (!u.parent anchor) -- +(5mm,0) |- (.child anchor)\forestoption{edge label};
      },
      parent anchor=east,
      child anchor=west,
      l sep=10mm,
      tier/.wrap pgfmath arg={tier #1}{level()},
      edge={very thick},
      fill=white,
      rounded corners=2pt,
      drop shadow,
    }
    [Tangency problems
            [Linear tangency,edge=dotted
                [Bitangents to plane quartics \cite{LV21}]
            ]
            [Quadric tangency,edge=dotted
                [Circles of Apollonius (McK22)]
                [Spheres tangent to four spheres,edge=dotted]
            ]
        ]
    \end{forest}
    \caption{Tangency problems}\label{fig:tangency problems}
\end{figure}
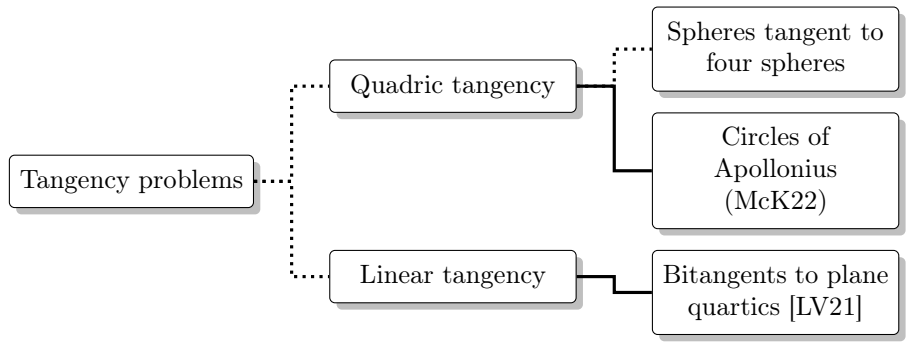

\begin{figure}[p]
\rotatebox{-90}{
    \begin{forest}
    for tree={
      if level=0{align=center}{
        align={@{}C{30mm}@{}},
      },
      grow=east,
      draw,
      font=\footnotesize,
      edge path={
        \noexpand\path [draw, \forestoption{edge}] (!u.parent anchor) -- +(5mm,0) |- (.child anchor)\forestoption{edge label};
      },
      parent anchor=east,
      child anchor=west,
      l sep=10mm,
      tier/.wrap pgfmath arg={tier #1}{level()},
      edge={very thick},
      fill=white,
      rounded corners=2pt,
      drop shadow,
    }
    [Local complete intersections \cite{McK21}
        [Tangency problems
            [Linear tangency,edge=dotted
                [Bitangents to plane quartics \cite{LV21}]
            ]
            [Quadric tangency,edge=dotted
                [Circles of Apollonius (McK22)]
                [Spheres tangent to four spheres,edge=dotted]
            ]
        ]
        [Rational curves on hypersurfaces
            [Lines on hypersurfaces,edge=dotted
                [Lines on cubic surfaces \cite{KW21}]
                [Lines on quintic threefolds \cite{Pau20}]
            ]
            [Conics on quintic threefolds (BMP),edge=dotted]
            [Twisted cubics on hypersurfaces \cite{LP22},edge=dotted]
        ]
        [Varieties meeting specified locus
            [Curves meeting linear spaces,edge=dotted
                [Rational curves through points (KLSW),edge=dotted]
                [Curves through lines,edge=dotted
                    [Conics through lines \cite{DGGM21}]
                    [Plane curves through lines (MP)]
                    [Space curves through lines,edge=dotted]
                ]
            ]
            [Linear spaces meeting linear spaces \cite{Bra22},edge=dotted
                [Lines through hyperplanes \cite{SW21}]
            ]
            [Linear series on hyperelliptic curves \cite{CDH20},edge=dotted]
        ]
    ]
    \end{forest}}
    \caption{Local phylogeny of enriched enumerative problems}\label{fig:phylogeny}
\end{figure}
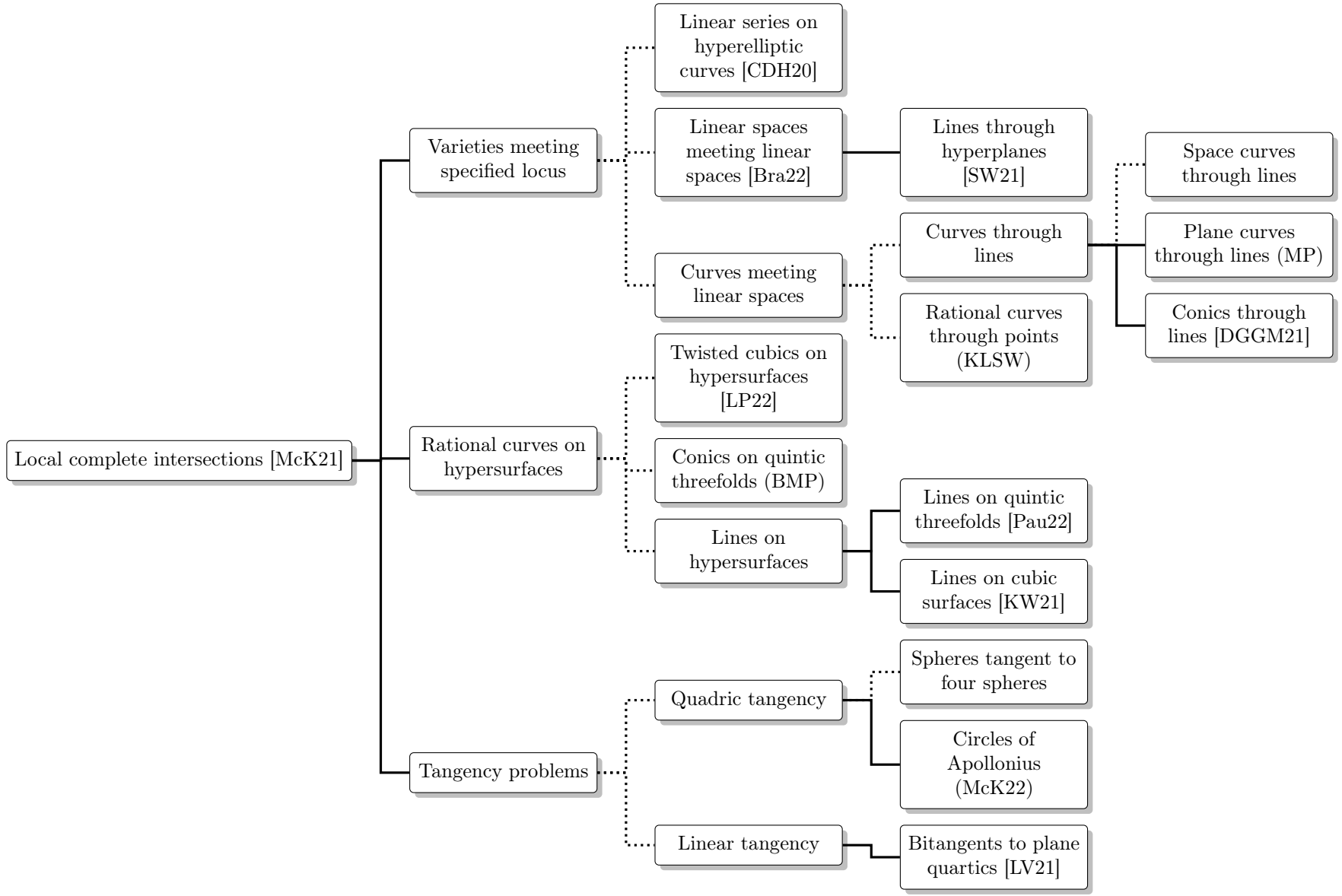
\newpage

\bibliography{apollonius}{}
\bibliographystyle{alpha}
\end{document}